\documentclass[11pt,letter]{article}

\usepackage{color,graphicx}
\usepackage{citesort}
\usepackage{subfigure}
\usepackage{amsmath,amssymb,amsfonts}

\newtheorem{theorem}{Theorem}

\newtheorem{lemma}[theorem]{Lemma}
\newtheorem{remark}{Remark}
\newtheorem{proposition}{Proposition}
\newenvironment{proof}[1][Proof]{\noindent\textbf{#1.} }{\ \rule{0.5em}{0.5em}}

\begin{document}

%\title[Gradual time reversal in TAT/PAT]{Gradual time reversal in thermo-
%and photo- acoustic tomography within a resonant cavity}

\title{Gradual time reversal in thermo-
and photo- acoustic tomography within a resonant cavity}

\author{B. Holman and L. Kunyansky}

\maketitle

\begin{abstract}
Thermo- and photo- acoustic tomography require reconstructing initial
acoustic pressure in a body from time series of pressure measured on a
surface surrounding the body. For the classical case of free space wave
propagation, various reconstruction techniques are well known. However, some
novel measurement schemes place the object of interest between reflecting
walls that form a de facto resonant cavity. In this case, known methods
(including the popular time reversal algorithm) cannot be used. The inverse
problem involving reflecting walls can be solved by the \emph{gradual time
reversal} method we propose here. It consists in solving back in time on the
interval $[0,T]$ the initial/boundary value problem for the wave equation,
with the Dirichlet boundary data multiplied by a smooth cut-off function. If
$T$ is sufficiently large one obtains a good approximation to the initial
pressure; in the limit of large $T$ such an approximation converges (under
certain conditions) to the exact solution.
\end{abstract}

%\author{B. Holman and L. Kunyansky}
%
%\address{
%Department of Mathematics,
%University of Arizona,
%Tucson, AZ 85721,
%USA
%}
%
%\ead{leonk@math.arizona.edu}
%
%
%\ams{44A12, 92C55, 65R32}
%
%\textit{Keywords}: photoacoustic tomography, thermoacoustic tomography, time
%reversal, resonant cavity, reflecting walls, wave equation

\section*{Introduction}

Photoacoustic tomography (PAT)~\cite{KrugerPAT,Oraev94,Beard2011} and
thermoacoustic tomography (TAT)~\cite{KrugerTAT,WangCRC} are based on the
thermoacoustic effect: when a material is heated it expands. To perform
measurements, a biological object is submerged in water (or hydroacoustic
gel) and is illuminated with a short electromagnetic pulse that heats the
tissue. The resulting thermoacoustic expansion generates an outgoing
acoustic wave, whose pressure is measured on a surface (completely or
partially) surrounding the object. Next, an inverse problem is solved in
order to image the initial acoustic pressure within the object. This
pressure is closely related, in particular, to the blood content in tissues.
Blood vessels an cancerous tumors produce much higher pressure; accordingly,
TAT and PAT are effective for cancer detection and for imaging vasculature
in small animals.

During the past decade, the mathematical foundations of TAT and PAT have
been well investigated. Significant achievements include, in particular,
results on general solvability and stability of the underlying inverse
problem ~\cite{AK,AKQ,AQ,AmbKu,QuiRull,AFK,AN,SU}, explicit inversion
formulas ~\cite%
{Finch04,Finch07,Kunyansky,Kun-cube,Nguyen,MXW2,Salman,Pala,Nat12,QSU}, and
efficient computational methods ~\cite%
{Norton1,Norton2,Kun-ser,XuWang04,HKN,burg-exac-appro}. One of the active
areas of current research is the so-called quantitative PAT (QPAT) \cite%
{QPAT_CLA,QPAT_PCA,QPAT_TPC} which aims to recover, in addition to the
initial pressure, optical properties of the tissue (e.g., Gr\"uneisen
coefficient) and the fluency of electromagnetic radiation as it propagates
through inhomogeneous tissue.

However, practically all existing theory of TAT/PAT is based on the
assumption that acoustic waves propagate in free space, and that reflections
from detectors and the walls of the water tank can be either neglected or
gated out. In the latter case, acoustic pressure $p(t,x)$ within the object
vanishes quite fast (in a finite time if the speed of sound is constant
within the domain). In this case the inverse problem of TAT/PAT can be
solved by ``time reversal", i.e. by solving back in time the
initial/boundary value problem for the wave equation, with the Dirichlet
boundary values equal to the measured data (see, for example, \cite%
{XuWang04,HKN,burg-exac-appro}). Vanishing (or sufficient decrease) of the
pressure in finite time $T$ allows one to initialize this process by setting
$p(T,x)$ and its time derivative $p_{t}(T,x)$ to zero within the domain~$%
\Omega $. Time reversal yields a theoretically exact reconstruction; it can
be implemented in a general closed domain (even with a variable speed of
sound) using finite differences; it can also be realized (for simple
domains) using the method of separation of variables, or, for certain
geometries, replaced by equivalent explicit backprojection formulas. Other
reconstruction algorithms, although not related directly to time reversal,
also require the vanishing of pressure.

However, free space propagation cannot always be used as a valid model. For
example, one of the most advanced PAT acquisition schemes (developed by
researchers from the University College London~\cite{Ell-Cox-apparatus})
uses optically scanned planar glass surfaces for the detection of acoustic
signals. Such surfaces act as (almost) perfect acoustic mirrors. If the
object is surrounded by such reflecting detectors (or by a combination of
detectors and acoustic mirrors), wave propagation occurs in a resonant
cavity. It involves multiple reflections of waves from the walls, and, if
the dissipation of waves is neglected, the acoustic oscillations will never
end. Traditional time reversal and other existing techniques are not
applicable in this case; new reconstruction algorithms need to be developed
for TAT/PAT\ within resonant cavities.

In \cite{Cox} the authors jointly with B. T. Cox developed such an algorithm
for a rectangular resonant cavity. That method is based on the Fast Fourier
transform and is computationally very efficient; however, it is not easily
exteneded to other geometries, and it cannot handle the case of variable
speed of sound. (Other approaches to the inverse problem within resonant
cavity include \cite{Cox2007,Cox2009,WangYang07}; in \cite{Ammari} an
approximate solution is obtained assuming that the sources of sound wave are
small inclusions.)

In the present paper we investigate the possibility of solving the inverse
problem of TAT/PAT in a resonant cavity by a modified time reversal
technique. Here we understand time reversal in a general sense, without
specifying the particular computational technique used to solve the
underlying initial/boundary value problem numerically (although in our
simulations we used an algorithm based of finite differences). Since (in the
idealized setting) the acoustic energy is preserved within the domain,
initializing classical time reversal by setting $p_{t}(T,x)=p(T,x)=0$ for
any value of $T$ would introduce an error of the same order of magnitude as
the initial pressure we seek. Instead, we propose a version of time reversal
where the boundary data are multiplied by a smooth cut-off function equal to
1 at times $t$ close to 0 and vanishing at $t=T$ together with all (or, at
least, several) derivatives. This technique, which we call \emph{gradual
time reversal,} can be initialized by $p_{t}(T,x)=p(T,x)=0$. As we show in
the paper, such an approach yields a good approximation to the sought
initial pressure $p(0,x)$ if $T$ is sufficiently large. Moreover, under
rather generic conditions this approximation converges to $p(0,x)$ in the
limit $T\rightarrow \infty .$

The rest of the paper is organized as follows. In the next section we give a
precise formulation of the problem. Section \ref{S:general} presents the
gradual time reversal algorithm and the theorems establishing weak
convergence of this technique under some rather generic conditions. In
Section \ref{S:particular} we consider circular and rectangular domains
where stronger convergence results can be obtained; in particular, we prove
strong convergence in $H^{1}(\Omega )$ of gradual time reversal in a
circular domain. Several results of numerical simulations are also presented
in the latter section to demonstrate the practicality of the present method.
The paper is concluded with an Appendix containing an auxiliary theorem on
relative spacing of the zeros of Bessel functions and their derivatives
(needed in Section~\ref{S:particular} to analyze convergence of graduate
time reversal in a circular domain).

\section{Formulation of the problem}

The pressure differential $u(t,x)$ within a reverberant cavity is a solution
to the following initial/boundary value problem:%
\begin{equation}
\left\{
\begin{array}{ccc}
\frac{1}{c^{2}(x)}\frac{\partial^{2}}{\partial t^{2}}u(t,x)=\Delta u(t,x), &
x\in\Omega, & t\in\lbrack0,\infty), \\
u(0,x)=f(x),\quad\frac{\partial u}{\partial t}(0,x)=0, & x\in\Omega, &  \\
\frac{\partial u(t,z)}{\partial\mathbf{n}}=0, & z\in\Sigma, & t\in
\lbrack0,\infty).%
\end{array}
\right.  \label{E:forward-system}
\end{equation}
where $\Sigma$ is the boundary of the bounded domain $\Omega\subset \mathbb{R%
}^{d}$ formed by the walls of the cavity, $c(x)$ is the known speed of sound
within the cavity, $\mathbf{n}$ is the exterior normal to $\Omega$, and $%
\frac{\partial u}{\partial\mathbf{n}}$ is the normal derivative of $u$. The
measured data $U(t,z)$ coincides with $u(t,z)$ on a part of the boundary $%
\Sigma_{1}\subseteq$ $\Sigma$ ($\Sigma_{1}$ may in some cases coincide with
the whole $\Sigma)$:%
\begin{equation}
U(t,z)=u(t,z),\quad z\in\Sigma_{1},\quad t\in\lbrack0,\infty).  \notag
\end{equation}
Our goal is to reconstruct the initial condition $f(x)$ from $U(t,z)$.

In the setting of traditional PAT/TAT, wave propagation occurs in the whole
space (the reflection from the detectors is assumed to be negligible). In
the simplest case of 3D wave propagation with constant speed of sound the
pressure vanishes in $\bar{\Omega}$ after a finite time $t=T$. If, in
addition, $\Sigma _{1}=\Sigma$, the initial pressure $f(x)$ can be found by
time reversal, i.e. by solving the wave equation in $Q_{T}\equiv \Omega
\times \lbrack 0,T]$ back in time from $t=T$ to $t=0.$ \ One imposes on such
a solution $\tilde{u}$ initial conditions $\tilde{u}(T,x)=0$ and $\frac{%
\partial \tilde{u}}{\partial t}(T,x)=0,$ and forces $\tilde{u}(t,z)$ on $%
\Sigma \times \lbrack 0,T]$ to be equal to the measured data $U(t,z).$ Then,
so-computed values of $\tilde{u}(0,x)$ coincide with $f(x).$ This method
also works in 2D and/or if the speed of sound is variable (but
non-trapping), in the limit of a large $T.$

However, in the case of perfectly reflecting boundaries we consider here,
the energy of the acoustic waves is preserved, and $u(t,x)$ remains of the
same order of magnitude for all values of $t\in \lbrack 0,\infty ).$ Since
values of the pressure (and its time derivative) inside $\Omega $ cannot be
measured, there is no accurate way of initializing time reversal. Simply
replacing the unknown values $u(T,x)$ and $\frac{\partial u}{\partial t}%
(T,x) $ by zero would introduce an error proportional to the energy of the
acoustic waves at the time $T,$ which will propagate toward $t=0$ and create
artifacts roughly of the same order of magnitude as $f(x).$

Below, we show that a good approximation to $f(x)$ can be obtained by
solving a modified time reversal problem; %for the lack of a better term
we will call this technique \emph{gradual time reversal}.

\section{Gradual time reversal: general considerations\label{S:general}}

Let us introduce an infinitely smooth cut-off function $\alpha (t)$ defined
on $[0,1],$ identically equal to $1$ within some neighborhood of $0,$ and
vanishing with all its derivatives at $1.$ Gradual time reversal consists in
solving back in time the initial/boundary value problem for the wave
equation with zero initial conditions at $t=T,$ and boundary conditions
equal to $U(t,z)\alpha (\varepsilon t)$ where $\varepsilon =1/T$ on $\Sigma
_{1}.$ On the rest of the boundary $\Sigma _{2}\equiv \Sigma \setminus
\Sigma _{1}$, where Dirichlet data are not available, we impose zero Neumann
boundary conditions:%
\begin{equation}
\left\{
\begin{array}{lccc}
\frac{1}{c^{2}(x)}\frac{\partial ^{2}v_{\varepsilon }}{\partial t^{2}}%
(t,x)=\Delta v(t,x), & x\in \Omega , & t\in \lbrack 0,T], & T=1/\varepsilon ,
\\
v_{\varepsilon }(T,x)=0,\quad \frac{\partial v_{\varepsilon }}{\partial t}%
(T,x)=0, & x\in \Omega , &  &  \\
v_{\varepsilon }(t,z)=U(t,z)\alpha (\varepsilon t), & z\in \Sigma _{1}, &
t\in \lbrack 0,T], &  \\
\frac{\partial v(t,z)}{\partial \mathbf{n}}=0 & z\in \Sigma _{2}, & t\in
\lbrack 0,T]. &
\end{array}%
\right.  \label{E:inversesystem}
\end{equation}%
As we show below, for sufficiently small values of $\varepsilon $ (or,
equivalently, large values of $T$)$,$ $v_{\varepsilon }(0,z)$ is a good
approximation to $f(x),$ and in fact, under certain conditions $%
v_{\varepsilon }(0,z)$ converges to $f(x)$, as $\varepsilon \rightarrow 0.$

\subsection{Some facts about the forward problem}

In the rest of the paper we will assume that the speed of sound $c(x)$ is \
a known, twice differentiable function bounded from above and below in $%
\Omega $:%
\begin{equation}
0<c_{\min }\leq c(x)\leq c_{\max },\quad \forall x\in \bar{\Omega},  \notag
\end{equation}%
and the boundary $\Sigma $ is piece-wise smooth. The initial condition $f$
is assumed to be compactly supported within $\Omega $ and be an element of
the Hilbert space $H^{1}(\Omega )$ with the inner product $[.,.]_{H^{1}}$
and the norm $||h||_{H^{1}}$ defined, for any $\forall g,h\in H^{1}(\Omega )$%
, as follows
\begin{equation}
\lbrack g,h]_{H^{1}}\equiv \int\limits_{\Omega }\left\{ \frac{1}{c^{2}(x)}%
g(x)\overline{h(x)}+\nabla g(x)\cdot \nabla \overline{h(x)}\right\} dx,\quad
||h||_{H^{1}}\equiv \sqrt{\lbrack h,h]_{H^{1}}}.  \notag
\end{equation}

%\begin{align*}
%a  & =b\\
%c  & =d
%\end{align*}

Since a classical solution of the wave equation may not exist under these
assumptions, we will understand the wave equation~(\ref{E:forward-system},
first line) in the weak sense:
\begin{equation}
\left( \frac{1}{c^{2}}\frac{\partial ^{2}u}{\partial t^{2}}(t,\cdot ),\eta
(\cdot )\right) _{L_{2}}+\left( \nabla u(t,\cdot ),\nabla \eta (\cdot
)\right) _{L_{2}}=0,\text{ }\forall \eta (x)\in C_{0}^{\infty }(\Omega
),\quad t\in (0,T).  \label{E:week-u}
\end{equation}%
where $(.,.)_{L_{2}}$ stands for the$\ $inner product in $L_{2}(\Omega )$
when applied to scalar functions:%
\begin{equation}
(g,h)_{L_{2}}\equiv \int\limits_{\Omega }g(x)\overline{h(x)}dx,\quad \forall
g,h\in L_{2}(\Omega );  \notag
\end{equation}%
expression $\left( \nabla g,\nabla h\right) _{L_{2}}$ is understood as
follows%
\begin{equation}
(\nabla g,\nabla h)_{L_{2}}\equiv \int\limits_{\Omega }\nabla g(x)\cdot
\overline{\nabla h(x)}dx,\quad \forall g,h\in H^{1}(\Omega ).
\label{E:inner-grad}
\end{equation}

It is known \cite{Ladyzh} that under these conditions there exists a unique
solution $u(t,x)$ of~(\ref{E:week-u}) in the class $C(0,T;H^{1}(\Omega ))$
on $Q_{T}\equiv (0,T)\times \Omega ,$ whose time derivative $\frac{\partial u%
}{\partial t}(t,x)$ and the first order space derivatives $\frac{\partial u}{%
\partial x_{j}}(t,x)$ are $L_{2}$ functions on $Q_{T}.$

Using separation of variables, this solution can be found in the form of a
generalized Fourier series. In order to accomplish this, one finds the
eigenfunctions $\varphi _{n}(x)$ of the weighted Neumann Laplacian on $%
\Omega $:%
\begin{equation}
-\Delta \varphi _{n}(x)=\frac{1}{c^{2}(x)}\lambda _{n}^{2}\varphi
_{n}(x),\quad \left. \frac{\partial \varphi _{n}(z)}{\partial \mathbf{n}}%
\right\vert _{\partial \Omega }=0, \quad n=1,2,3,..., \   \notag
\end{equation}%
where $\lambda _{n}^{2}$ are the corresponding eigenvalues, in
non-decreasing order with $\lambda _{1}=0$ and $\lambda _{2}>0.$ In general,
eigenfunctions $\varphi _{n}(x)$ can be found in the class $H^{1}(\Omega );$
they are pair-wise orthogonal with respect to the weighted $L_{2}$ inner
product%
\begin{equation}
<g,h>_{c^{-2}}\equiv \int\limits_{\Omega }\frac{1}{c^{2}(x)}g(x)\overline{%
h(x)}dx,\quad \forall g,h\in L_{2}(\Omega ,c^{-2}(x)).  \label{E:innerpr}
\end{equation}%
so that
\begin{equation}
<\varphi _{l},\varphi _{n}>_{c^{-2}}=0,\text{ if }l\neq n.  \label{E:ortho1}
\end{equation}%
Assume that these eigenfunctions are normalized with respect to the weighted
$L_{2}$ norm:%
\begin{equation}
||\varphi _{l}||_{c^{-2}}\equiv \sqrt{<\varphi _{l},\varphi _{l}>_{c^{-2}}}%
=1,\text{\qquad }l=1,2,3,....  \label{E:ortho2}
\end{equation}%
It is known that, with such normalization, these eigenfunctions are also
orthogonal with respect to the inner product given by equation (\ref%
{E:inner-grad}):
\begin{equation}
(\nabla \varphi _{l},\nabla \varphi _{n})_{L_{2}}=0,\text{ if }l\neq
n,\qquad (\nabla \varphi _{l},\nabla \varphi _{l})_{L_{2}}=\lambda _{l}^{2},%
\text{ }l=1,2,3,....,  \label{E:ortho4}
\end{equation}

By utilizing $\varphi _{n}$'s, the weak solution of (\ref{E:week-u}) can be
found in the form of the series%
\begin{equation}
u(t,x)=\sum_{n=1}^{\infty }u_{n}\varphi _{n}(x)\cos (\lambda _{n}t),
\label{E:Useries}
\end{equation}%
with Fourier coefficients $u_{n}$ found from the initial condition%
\begin{equation}
u_{n}=<f,\varphi _{n}>_{c^{-2}},\quad n=1,2,3,...  \notag
\end{equation}%
We note the Parseval's identity and a related inequality%
\begin{eqnarray}
||u(0,\cdot )||_{c^{-2}}^{2} &=&\sum_{n=1}^{\infty }|u_{n}|^{2},
\label{E:Parseval} \\
||u(0,t)||_{c^{-2}}^{2} &=&\sum_{n=1}^{\infty }|u_{n}\cos (\lambda
_{n}t)|^{2}\leq \sum_{n=1}^{\infty }|u_{n}|^{2}.  \label{E:Parseval-a}
\end{eqnarray}%
Also, since $f\in $ $H^{1}(\Omega ),$ the following bound on $u_{n}$ holds
with some constant $E_{0}.$%
\begin{equation}
\int\limits_{\Omega }|\nabla f|^{2}dx=\sum_{n=1}^{\infty }\lambda
_{n}^{2}|u_{n}|^{2}\equiv E_{0}<||f||_{H^{1}}^{2}.  \label{E:E0}
\end{equation}%
Now, if one defines energy $E(t)$ by the formula
\begin{equation}
E(t)=\int\limits_{\Omega }\left[ \frac{1}{c^{2}(x)}\left\vert \frac{\partial
u}{\partial t}(t,x)\right\vert ^{2}+|\nabla u(t,x)|^{2}\right] dx,
\label{E:energy}
\end{equation}%
this energy will be conserved, i.e.%
\begin{equation}
E(t)=E(0)=E_{0}.  \notag
\end{equation}%
This can be easily verified either by substituting $\eta =\frac{\partial u}{%
\partial t}$ into (\ref{E:week-u}) and integrating from $0$ to $t$ in time,
or by substituting the series representation (\ref{E:Useries}) into (\ref%
{E:week-u}) for both $u$ and $\eta =u,$ and using the orthogonality
relations (\ref{E:ortho1})-(\ref{E:ortho4}).

Let us also find uniform (in $t$) bounds on the solution $u$ and its time
derivative. A bound on $\frac{\partial u}{\partial t}(t,x)$ for any $t,$ in
the weighted $L_{2}$ norm follows from~ (\ref{E:energy}):%
\begin{equation}
\left\Vert \frac{\partial u}{\partial t}\right\Vert _{c^{-2}}^{2}\equiv
\int\limits_{\Omega }\frac{1}{c^{2}(x)}|u_{t}(t,x)|^{2}dx\leq
E_{0}=\int\limits_{\Omega }|\nabla f|^{2}dx\leq ||f||_{H^{1}}^{2},\quad t\in
\lbrack 0,\infty ).  \label{E:bounduprime}
\end{equation}%
In turn, $u(t)$ can be bound by combining (\ref{E:Parseval-a}) and (\ref%
{E:E0}):
\begin{align}
||u||_{c^{-2}}^{2}& \leq \sum_{n=1}^{\infty }|u_{n}|^{2}\leq
|u_{1}|^{2}+\sum_{n=2}^{\infty }|u_{n}|^{2}\leq ||f||_{c^{-2}}^{2}+\frac{1}{%
\lambda _{2}^{2}}\sum_{n=2}^{\infty }\lambda _{n}^{2}|u_{n}|^{2}  \notag \\
& \leq ||f||_{c^{-2}}^{2}+\frac{1}{\lambda _{2}^{2}}\int\limits_{\Omega
}|\nabla f|^{2}dx\leq \left( 1+\frac{1}{\lambda _{2}^{2}}\right)
||f||_{H^{1}}^{2},\qquad t\in \lbrack 0,\infty ).  \label{E:boundu}
\end{align}

\subsection{Convergence of gradual time reversal\label{S:Theory}}

We would like to show that, under certain conditions, the solution $%
v_{\varepsilon}(0,x)$ of the gradual time reversal problem (\ref%
{E:inversesystem}) converges to $f(x)$ as $\varepsilon\rightarrow0.$ We will
represent $v_{\varepsilon}(t,x)$ as a sum of two functions%
\begin{equation}
v_{\varepsilon}(t,x)=u(t,x)\alpha(\varepsilon t)+w_{\varepsilon}(t,x),
\label{E:sumoftwo}
\end{equation}
where $u(t,x)$ is the (unknown) solution of the forward problem. Since%
\begin{equation}
v_{\varepsilon}(0,x)=u(0,x)+w_{\varepsilon}(0,x)=f(x)+w_{\varepsilon}(0,x),
\notag
\end{equation}
function $w_{\varepsilon}(0,x)$ represents the error introduced by gradual
time reversal into the reconstruction; we would like to show that it becomes
small as $\varepsilon\rightarrow0.$

The first term in the right hand side of (\ref{E:sumoftwo}) accounts for the
known Dirichlet boundary values on $\Sigma _{1}$, so that $w_{\varepsilon
}(t,x)$ satisfies zero Dirichlet conditions on $\partial \Omega _{1}$ and
the zero Neumann boundary values on $\Sigma _{2},$ for all values of $t$:%
\begin{equation}
\left\{
\begin{array}{ccc}
w_{\varepsilon }(t,z)=0, & z\in \Sigma _{1}, &  \\
\frac{\partial w_{\varepsilon }(t,z)}{\partial \mathbf{n}}=0, & z\in \Sigma
_{2.} &
\end{array}%
\right. ,\quad t\in \lbrack 0,T]  \label{E:boundary-w}
\end{equation}%
Also, since the derivatives of $\alpha (\varepsilon t)$ vanish at $t=T,$%
\begin{equation}
w_{\varepsilon }(T,x)=0,\quad \frac{\partial w_{\varepsilon }(T,x)}{\partial
t}=0,\quad x\in \Omega .  \label{E:initial-w}
\end{equation}

By substituting (\ref{E:sumoftwo}) into the wave equation~(formula (\ref%
{E:inversesystem}), first line) and taking into account~(\ref%
{E:forward-system}) we obtain%
\begin{align}
\frac{1}{c^{2}(x)}\frac{\partial ^{2}w_{\varepsilon }}{\partial t^{2}}%
(t,x)-\Delta w(t,x)& =\frac{1}{c^{2}(x)}F_{\varepsilon }(t,x),
\label{E:wave-for-w} \\
F_{\varepsilon }(t,x)& \equiv -\varepsilon \left( 2\alpha ^{\prime
}(\varepsilon t)\frac{\partial u}{\partial t}(t,x)+\varepsilon \alpha
^{\prime \prime }(\varepsilon t)u(t,x)\right) .  \label{E:RHS}
\end{align}%
It follows that $w_{\varepsilon }(t,x)$ solves the initial/boundary value
problem for the wave equation~(\ref{E:wave-for-w}), (\ref{E:boundary-w}), (%
\ref{E:initial-w}) with the right hand side given by (\ref{E:RHS}). Since $u,%
\frac{\partial u}{\partial t}\in L^{2}(Q_{T})$, the right hand side $%
F_{\varepsilon }(t,x)$ is also an $L_{2}$ function on $Q_{T},$ and the wave
equation (\ref{E:wave-for-w}) should be understood in the weak sense:%
\begin{equation}
\left( \frac{1}{c^{2}}\frac{\partial ^{2}w_{\varepsilon }}{\partial t^{2}}%
(t,\cdot ),\eta (\cdot )\right) _{L_{2}}+\left( \nabla w_{\varepsilon
}(t,\cdot ),\nabla \eta (\cdot )\right) _{L_{2}}=\left( \frac{1}{c^{2}(x)}%
F_{\varepsilon }(t,\cdot ),\eta (\cdot )\right) _{L_{2}}=\left\langle
F_{\varepsilon },\eta \right\rangle _{c^{-2}}  \notag
\end{equation}%
for all values of $t\in (0,T)$ and for all $\eta (x)\in C^{\infty }(\bar{%
\Omega})$ and vanishing on $\Sigma _{1}.$

\subsubsection{Boundedness of $w_{\protect\varepsilon}$}

Our first step is to show that the error $w_{\varepsilon }(0,x)$ remains
bounded as $\varepsilon \rightarrow 0$ (or, what's the same, as $%
T\rightarrow \infty ).$ It is known \cite{Ladyzh} that the unique solution $%
w_{\varepsilon }(t,x)$ of the initial/boundary value problem of our type
with an $L_{2}$ right hand side can be found in the class $%
H^{1}(0,T;H^{1}(\Omega )),$ and that the first time- and space- derivatives
of $w_{\varepsilon }$ are $L_{2}$ functions on $Q_{T}.$ Moreover, using
separation of variables this solution can be found in the the form of a
generalized Fourier series. To this end one utilizes the eigenfunctions $%
\psi _{k}(x)$, $k=1,2,3,...$ of the weighted Laplacian on $\Omega $ with
mixed boundary conditions, with the corresponding eigenvalues $\nu _{k}^{2},$
$k=1,2,3,...$ :%
\begin{align}
-\Delta \psi _{k}(x)& =\frac{1}{c^{2}(x)}\nu _{k}^{2}\psi _{k}(x),  \notag \\
\left. \psi _{k}(z)\right\vert _{\Sigma _{1}}& =0,\quad \left. \frac{%
\partial \psi _{k}(z)}{\partial \mathbf{n}}\right\vert _{\Sigma _{2}}=0.
\label{E:mixedBC}
\end{align}%
Properties of these eigenfunctions are similar to those of the Neumann
eigenfunctions: they exist in $H^{1}(\Omega )$ and satisfy the orthogonality
conditions:
\begin{equation}
<\psi _{l},\psi _{n}>_{c^{-2}}=0,\text{ if }l\neq n.  \label{E:ortho11}
\end{equation}%
We again assume that $\psi _{n}$'s are normalized with respect to the
weighted $L_{2}$ norm, i.e.%
\begin{equation}
||\psi _{k}||_{c^{-2}}=1,\text{ }k=1,2,3,....,  \label{E:ortho12}
\end{equation}%
and we will use the fact that the gradients of these eigenfunctions are
orthogonal with respect to the inner product (\ref{E:inner-grad}):
\begin{equation}
(\nabla \psi _{l},\nabla \psi _{n})_{L_{2}}=0,\text{ if }l\neq n,\qquad
(\nabla \psi _{l},\nabla \psi _{l})_{L_{2}}=\lambda _{l}^{2},\text{ }%
l=1,2,3,....  \notag
\end{equation}

Now, $w_{\varepsilon}(t,x)$ can be represented in the form of the following
series:%
\begin{align}
w_{\varepsilon}(t,x) & =\sum_{k=1}^{\infty}w_{k}(t)\psi_{k}(x),
\label{E:Dirichlet-ser} \\
w_{k}(t) & =<w_{\varepsilon}(t,x),\psi_{k}>_{c^{-2}},\quad k=1,2,3,...,\quad
t\in\lbrack0,\infty).  \label{E:Dirichlet-coef}
\end{align}
(In the above formula Fourier coefficients $w_{k}(t)$ depend on $%
\varepsilon; $ however, for brevity this dependence is not reflected in our
notation.) Due to orthogonality of $\psi_{k}$ (in the sense of (\ref%
{E:ortho11})) each of the coefficients $w_{k}(t)$ satisfies the differential
equation
\begin{align}
w_{k}^{\prime\prime}(t)+\nu_{k}^{2}w_{k}(t) & =F_{k}(t),\qquad
F_{k}(t)\equiv\left\langle F_{\varepsilon},\psi_{k}\right\rangle _{c^{-2}},
\label{E:diffeq} \\
w_{k}(T) & =0,\qquad w_{k}^{\prime}(T)=0.  \notag
\end{align}
This equation is also understood in the weak sense. \ The causal Green's
functions $\Phi_{k}(t)$ of these equations are equal to $\frac{\sin(\nu_{k}t)%
}{\nu_{k}},$ so that%
\begin{equation}
\nu_{k}w_{k}(0)=\int\limits_{0}^{T}F_{k}(\tau)\sin(\nu_{k}\tau)d\tau,\qquad
w_{k}^{\prime}(0)=\int\limits_{0}^{T}F_{k}(\tau)\cos(\nu_{k}\tau )d\tau.
\label{E:solu1}
\end{equation}
Then, by Cauchy-Schwarz inequality,%
\begin{equation}
|\nu_{k}w_{k}(0)|^{2}\leq T\int\limits_{0}^{T}|F_{k}(\tau)|^{2}d\tau
,\qquad|w_{k}^{\prime}(0)|^{2}\leq
T\int\limits_{0}^{T}|F_{k}(\tau)|^{2}d\tau.  \label{E:Schwartz}
\end{equation}

Due to the boundedness of $\frac{\partial w_{\varepsilon }}{\partial t}$ and
$\nabla w_{\varepsilon }$ in $L_{2}(\Omega ),$ one can define the energy of
the solution $w_{\varepsilon }$ by a formula similar to (\ref{E:energy})%
\begin{equation}
E_{w}(t)=\int\limits_{\Omega }\left[ \frac{1}{c^{2}(x)}\left\vert \frac{%
\partial w_{\varepsilon }}{\partial t}(t,x)\right\vert ^{2}+|\nabla
w_{\varepsilon }(t,x)|^{2}\right] dx.  \notag
\end{equation}%
Using series representation (\ref{E:Dirichlet-ser}) and the orthogonality of
$\psi _{k}$'s:%
\begin{equation}
E_{w}(t)=\sum_{k=1}^{\infty }\left( |w_{k}^{\prime }(t)|^{2}+|\nu
_{k}w_{k}(t)|^{2}\right) .  \notag
\end{equation}%
Let us substitute into this equation bounds on $\nu _{k}w_{k}(0)$ and on $%
w_{k}^{\prime }(0)$ (equation (\ref{E:Schwartz})) and apply Tonelli's
theorem and Parseval's identity:%
\begin{equation}
\frac{1}{2T}E_{w}(0)\leq \sum_{k=1}^{\infty }\int\limits_{0}^{T}|F_{k}(\tau
)|^{2}d\tau =\int\limits_{0}^{T}\sum_{k=1}^{\infty }|F_{k}(\tau )|^{2}d\tau
=\int\limits_{0}^{T}||F_{\varepsilon }(\tau ,\cdot )||_{c^{-2}}^{2}d\tau .
\label{E:energyineq0}
\end{equation}%
Using (\ref{E:RHS}) and recalling that $\varepsilon =1/T$ one obtains
\begin{equation}
||F_{\varepsilon }(t,\cdot )||_{c^{-2}}^{2}\leq \frac{1}{T^{2}}\left(
4|\alpha ^{\prime }(\varepsilon t)|^{2}\left\Vert \frac{\partial u}{\partial
t}(t,x)\right\Vert _{c^{-2}}^{2}+\frac{1}{T^{2}}|\alpha ^{\prime \prime
}(\varepsilon t)|^{2}||u(t,x)||_{c^{-2}}^{2}\right) .  \label{E:energyineq}
\end{equation}%
Let us assume that $T\geq 1,$ and that
\begin{equation}
\max_{s\in \lbrack 0,1]}\alpha ^{\prime }(s)=A_{1},\qquad \max_{s\in \lbrack
0,1]}\alpha ^{\prime \prime }(s)=A_{2}.  \label{E:maximums}
\end{equation}%
Combining (\ref{E:energyineq}), (\ref{E:maximums}) with bounds on $u$ and $%
\frac{\partial u}{\partial t}$ (equations~(\ref{E:boundu}) and (\ref%
{E:bounduprime})) we find a time-independent bound on $||F_{\varepsilon
}(t,\cdot )||_{c^{-2}}^{2}$:%
\begin{align}
||F_{\varepsilon }(t,\cdot )||_{c^{-2}}^{2}& \leq \frac{1}{T^{2}}\left(
4A_{1}^{2}\left\Vert \frac{\partial u}{\partial t}(t,x)\right\Vert
_{c^{-2}}^{2}+A_{2}^{2}||u(t,x)||_{c^{-2}}^{2}\right)  \notag \\
& \leq \frac{1}{T^{2}}\left( 4A_{1}^{2}+\left( 1+\frac{1}{\lambda _{2}^{2}}%
\right) A_{2}^{2}\right) \left\Vert f\right\Vert _{H^{1}}^{2}\leq \frac{%
C_{1}(\Omega ,\alpha )}{T^{2}}\left\Vert f\right\Vert _{H^{1}}^{2}
\label{E:energyineq3}
\end{align}%
with%
\begin{equation}
C_{1}(\Omega ,\alpha )=\left( 4A_{1}^{2}+\left( 1+\frac{1}{\lambda _{2}^{2}}%
\right) A_{2}^{2}\right) .  \notag
\end{equation}%
Finally, by substituting (\ref{E:energyineq3}) into (\ref{E:energyineq0}) we
obtain%
\begin{equation}
E_{w}(0)\leq 2T\left\Vert f\right\Vert _{H^{1}}^{2}\int\limits_{0}^{T}\frac{%
C_{1}(\Omega ,\alpha )}{T^{2}}d\tau =2C_{1}(\Omega ,\alpha )\left\Vert
f\right\Vert _{H^{1}}^{2}.  \notag
\end{equation}%
This allows us to obtain an estimate for $||w_{\varepsilon }(0,\cdot
)||_{H^{1}}^{2}.$ Indeed%
\begin{eqnarray*}
||w_{\varepsilon }||_{H^{1}}^{2} &=&||w_{\varepsilon
}||_{c^{-2}}^{2}+||\nabla w_{\varepsilon }||_{2}^{2}=\sum_{k=1}^{\infty
}\left( |w_{k}|^{2}+|\nu _{k}w_{k}|^{2}\right) \leq \left( \frac{1}{\nu
_{1}^{2}}+1\right) ||\nabla w_{\varepsilon }||_{2}^{2} \\
&\leq &\left( \frac{1}{\nu _{1}^{2}}+1\right) E_{w},
\end{eqnarray*}%
so that%
\begin{equation}
||w_{\varepsilon }(0,\cdot )||_{H^{1}}^{2}\leq \left( \frac{1}{\nu _{1}^{2}}%
+1\right) E_{w}(0)\leq 2C_{1}(\Omega ,\alpha )\left( \frac{1}{\nu _{1}^{2}}%
+1\right) \left\Vert f\right\Vert _{H^{1}}^{2},  \notag
\end{equation}%
which implies the following

\begin{proposition}
Under the assumptions on $\Omega,$ $\Sigma,$ and $f$ made previously, the
error $w_{\varepsilon}(0,\cdot)$ remains bounded in $H^{1}(\Omega)$
independently of $\varepsilon$ (or, equivalently, of $T):$%
\begin{equation}
||w_{\varepsilon}(0,\cdot)||_{H^{1}}\leq\sqrt{2C_{1}(\Omega,\alpha)\left(
\frac{1}{\nu_{1}^{2}}+1\right) }\left\Vert f\right\Vert _{H^{1}}.
\label{E:prop1}
\end{equation}
\end{proposition}

\subsubsection{Weak convergence\label{S:Weak}}

In this section we will show that $w_{k}(0)$ converge to 0 as $\varepsilon
\rightarrow 0.$ To this end, let us again consider differential equations~(%
\ref{E:diffeq}) on coefficients $w_{k}(t).$ The right hand sides $F_{k}(t)$
of these equations equal%
\begin{align*}
F_{k}(t)& =\left\langle F_{\varepsilon }(t,\cdot ),\psi _{k}(\cdot
)\right\rangle _{c^{-2}}=-\varepsilon \left\langle \left[ 2\alpha ^{\prime
}(\varepsilon t)\frac{\partial u}{\partial t}+\varepsilon \alpha ^{\prime
\prime }(\varepsilon t)u\right] ,\psi _{k}\right\rangle _{c^{-2}} \\
& =\varepsilon \left\langle \left[ 2\alpha ^{\prime }(\varepsilon
t)\sum_{n=0}^{\infty }u_{n}\varphi _{n}(x)\lambda _{n}\sin \lambda
_{n}t-\varepsilon \alpha ^{\prime \prime }(\varepsilon t)\sum_{n=0}^{\infty
}u_{n}\varphi _{n}(x)\cos \lambda _{n}t\right] ,\psi _{k}\right\rangle
_{c^{-2}} \\
& =\varepsilon \sum_{n=1}^{\infty }u_{n}\left[ 2\alpha ^{\prime
}(\varepsilon t)\lambda _{n}\sin (\lambda _{n}t)-\varepsilon \alpha ^{\prime
\prime }(\varepsilon t)\cos (\lambda _{n}t)\right] <\psi _{k},\varphi
_{n}>_{c^{-2}}.
\end{align*}%
Solutions of these equation given by (\ref{E:solu1}) can be re-written in
the form%
\begin{equation}
\nu _{k}w_{k}(0)=\sum_{n=1}^{\infty }u_{n}I_{n,k}(\varepsilon )<\psi
_{k},\varphi _{n}>_{c^{-2}}  \label{E:inprser}
\end{equation}%
with
\begin{align}
I_{n,k}(\varepsilon )& \equiv \varepsilon \int\limits_{0}^{1/\varepsilon }
\left[ 2\alpha ^{\prime }(\varepsilon t)\lambda _{n}\sin (\lambda
_{n}t)-\varepsilon \alpha ^{\prime \prime }(\varepsilon t)\cos (\lambda
_{n}t)\right] \sin (\nu _{k}t)dt  \notag \\
& =\int\limits_{0}^{1}\left[ 2\alpha ^{\prime }(\tau )\lambda _{n}\sin
(\lambda _{n}\tau /\varepsilon )-\varepsilon \alpha ^{\prime \prime }(\tau
)\cos (\lambda _{n}\tau /\varepsilon )\right] \sin (\nu _{k}\tau
/\varepsilon )d\tau  \notag \\
& =\lambda _{n}\int\limits_{0}^{1}\alpha ^{\prime }(\tau )\left[ \cos \left(
\frac{\lambda _{n}-\nu _{k}}{\varepsilon }\tau \right) -\cos \left( \frac{%
\lambda _{n}+\nu _{k}}{\varepsilon }\tau \right) \right] d\tau  \notag \\
& +\frac{\varepsilon }{2}\int\limits_{0}^{1}\alpha ^{\prime \prime }(\tau )%
\left[ \sin \left( \frac{\lambda _{n}-\nu _{k}}{\varepsilon }\tau \right)
-\sin \left( \frac{\lambda _{n}+\nu _{k}}{\varepsilon }\tau \right) \right]
d\tau .  \label{E:integrals}
\end{align}

Let us find bounds on $I_{n,k}(\varepsilon)$ in the generic case when the
eigenvalues of the Neumann Laplacian, and the Laplacian with the mixed
boundary conditions (\ref{E:mixedBC}) do not coincide, i.e.,
\begin{equation}
\lambda_{n}\neq\nu_{k},\quad\forall n,k.  \label{E:distinct}
\end{equation}

In order to bound the integrals in (\ref{E:integrals}), we extend function $%
\alpha ^{\prime }(\tau )$ evenly to the interval $[-1,1]$ and further to $%
(-\infty ,\infty )$ by zeros. Let us denote this extended function by $%
\alpha _{1}^{\ast }(\tau );$ it is infinitely smooth on $\mathbb{R}_{1}.$
Similarly, we extend \ $\alpha ^{\prime \prime }(\tau )$ in an odd fashion
to the interval $[-1,1]$ \ and further to $(-\infty ,\infty )$ by zeros, and
denote the resulting infinitely smooth function by $\alpha _{2}^{\ast }(\tau
).$ Then the integrals on the last line in (\ref{E:integrals}) are equal (up
to a constant factor) to the values of the Fourier transforms of $\alpha
_{1}^{\ast }(\tau )$ and $\alpha _{2}^{\ast }(\tau )$ at the frequencies $%
(\lambda _{n}-\nu _{k})/\varepsilon $ and $(\lambda _{n}+\nu
_{k})/\varepsilon .$
\begin{align*}
I_{n,k}(\varepsilon )& =\sqrt{2\pi }\lambda _{n}\left( \widehat{\alpha
_{1}^{\ast }}((\lambda _{n}-\nu _{k})/\varepsilon )-\widehat{\alpha
_{1}^{\ast }}((\lambda _{n}+\nu _{k})/\varepsilon )\right) \\
& +\sqrt{2\pi }i\frac{\varepsilon }{2}\left( \widehat{\alpha _{2}^{\ast }}%
((\lambda _{n}-\nu _{k})/\varepsilon )-\widehat{\alpha _{2}^{\ast }}%
((\lambda _{n}+\nu _{k})/\varepsilon )\right)
\end{align*}%
where the Fourier transform $\hat{h}(\xi )$ of function $h(x)$ is defined as%
\begin{equation}
\hat{h}(\xi )=\frac{1}{\sqrt{2\pi }}\int\limits_{-\infty }^{+\infty
}h(x)e^{-ix\xi }dx.  \notag
\end{equation}%
Since both $\alpha _{1}^{\ast }(\tau )$ and $\alpha _{2}^{\ast }(\tau )$ are
finitely supported and infinitely smooth on $\mathbb{R}_{1}$, for any
integer $M$ one can find a constant $B(M)$ such that%
\begin{equation}
\sqrt{2\pi }|\widehat{\alpha _{1}^{\ast }}(\xi )|<\frac{B(M)}{1+|\xi |^{M}},%
\text{ and }\sqrt{2\pi }|\widehat{\alpha _{2}^{\ast }}(\xi )|<\frac{B(M)}{%
1+|\xi |^{M}}.  \notag
\end{equation}%
Now $I_{n,k}(\varepsilon )$ can be bounded by the following expression
(assuming $\varepsilon <2$):%
\begin{align}
|I_{n,k}(\varepsilon )|& <\lambda _{n}\left( \frac{\varepsilon ^{M}B(M)}{%
\varepsilon ^{M}+|\lambda _{n}-\nu _{k}|^{M}}+\frac{\varepsilon ^{M}B(M)}{%
\varepsilon ^{M}+(\lambda _{n}+\nu _{k})^{M}}\right) +\frac{\varepsilon }{2}%
\left( \frac{\varepsilon ^{M}B(M)}{\varepsilon ^{M}+|\lambda _{n}-\nu
_{k}|^{M}}\right. +  \notag \\
& +\left. \frac{\varepsilon ^{M}B(M)}{\varepsilon ^{M}+(\lambda _{n}+\nu
_{k})^{M}}\right) <(\lambda _{n}+1)\frac{2\varepsilon ^{M}B(M)}{|\lambda
_{n}-\nu _{k}|^{M}}.  \label{E:bound}
\end{align}%
\newline
Inequality (\ref{E:bound}) combined with (\ref{E:inprser}) can be used to
find a bound on coefficients$~w_{k}(0)$:%
\begin{align}
|\nu _{k}w_{k}(0)|& \leq \sum_{n=1}^{\infty }|u_{n}|\,|(\psi _{k},\varphi
_{n})_{c^{-2}}|\,|I_{n,k}(\varepsilon )|\leq \sum_{n=1}^{\infty
}|u_{n}|\,(\lambda _{n}+1)\frac{2\varepsilon ^{M}B(M)}{|\lambda _{n}-\nu
_{k}|^{M}}\,  \label{E:good-estimate} \\
& \leq 2\varepsilon ^{M}B(M)\left( \sum_{n=1}^{\infty }|u_{n}|\frac{1}{%
|\lambda _{n}-\nu _{k}|^{M}}+\sum_{n=1}^{\infty }|\lambda _{n}u_{n}|\frac{1}{%
|\lambda _{n}-\nu _{k}|^{M}}\right)  \notag \\
& \leq 2\varepsilon ^{M}B(M)\left( ||f||_{c^{-2}}+||\nabla f||_{2}\right)
\sum_{n=1}^{\infty }\frac{1}{|\lambda _{n}-\nu _{k}|^{M}}  \notag \\
& =2\varepsilon ^{M}B(M)||f||_{H^{1}}\sum_{n=1}^{\infty }\frac{1}{|\lambda
_{n}-\nu _{k}|^{M}},  \label{E:nice-estimate}
\end{align}%
where we took into account that $|<\psi _{k},\varphi _{n}>_{c^{-2}}|$ cannot
exceed $1,|u_{n}|$ cannot exceed the weighted $L_{2}$ norm of $f(x),$ and $%
|u_{n}\lambda _{n}|$ is less or equal than the $L_{2}$ norm of $|\nabla f|.$

It is well known (see for example, \cite{Fleck} and references therein) that
the eigenvalues $\lambda_{n}$ grow without a bound as $n\rightarrow\infty;$
the asymptotic rate of growth is%
\begin{equation}
\lambda_{n}\sim C_{2}(\Omega,c(x))n^{\frac{2}{d}},  \notag
\end{equation}
where $C_{2}(\Omega,c(x))$ is a domain-dependent positive constant, and $d$
is the dimensionality of the space. This implies that for sufficiently large
values of $M$ (e.g., $M\geq d)$ the series in (\ref{E:nice-estimate})
converges. This, in turn, yields a convergence result for each $|w_{k}(0)|:$%
\begin{equation}
|w_{k}(0)|\leq C_{3}(M,k)||f||_{H^{1}}\varepsilon^{M}\underset{\varepsilon
\rightarrow0}{\rightarrow}0,  \notag
\end{equation}
with%
\begin{equation}
C_{3}(M,k)=2B(M)\frac{1}{\nu_{k}}\sum_{n=1}^{\infty}\frac{1}{%
|\lambda_{n}-\nu_{k}|^{M}}.  \notag
\end{equation}

On the other hand,%
\begin{equation}
\lbrack
w_{\varepsilon},\psi_{k}]_{H^{1}}=<w_{\varepsilon},\psi_{k}>_{c^{-2}}+(%
\nabla w_{\varepsilon},\nabla\psi_{k})_{L_{2}}=(1+\nu_{k}^{2})w_{k}(t).
\notag
\end{equation}
Therefore
\begin{equation}
|[w_{\varepsilon},\psi_{k}]_{H^{1}}|\leq
C_{4}(M,k)||f||_{H^{1}}\varepsilon^{M}\underset{\varepsilon\rightarrow0}{%
\rightarrow}0.  \label{E:converg}
\end{equation}
with
\begin{equation}
C_{4}(M,k)=2B(M)\frac{(1+\nu_{k}^{2})}{\nu_{k}}\sum_{n=1}^{\infty}\frac {1}{%
|\lambda_{n}-\nu_{k}|^{M}}.  \label{E:C4}
\end{equation}
We have thus proven the following

\begin{theorem}
Under the assumptions on $\Omega,$ $\Sigma,$ and $f$ made previously, and
under the condition (\ref{E:distinct}), the result $v_{\varepsilon}(0,\cdot)$
of gradual time reversal converges to $f$ weakly in $H^{1}(\Omega)$ as $%
\varepsilon\rightarrow0$ (or, equivalently, as $T\rightarrow\infty).$
\end{theorem}

\begin{proof}
The error $w_{\varepsilon}(0,x)=v_{\varepsilon}(0,x)-f(x)$ remains bounded
in $H^{1}(\Omega)$ (see (\ref{E:prop1})), and it satisfies (\ref{E:converg})
for all $\psi_{k}.$
\end{proof}

\begin{remark}
In general, decay of the coefficients $w_{k}(0)$ as $\varepsilon\rightarrow0$
is not uniform in $k.$ The factor $\frac{(1+\nu_{k}^{2})}{\nu_{k}}$ in (\ref%
{E:C4}) is growing as $k\rightarrow\infty,$ and there is no reason to expect
the sum of the series in (\ref{E:C4}) to decrease in $k$ in the general case.
\end{remark}

\subsubsection{The case of coinciding eigenvalues.}

Weak convergence was proven in the previous section by showing that all the
coefficients $I_{n,k}(\varepsilon )$ in (\ref{E:inprser}) converge to 0 as $%
\varepsilon \rightarrow 0,$in the case when the eigenvalues $\lambda _{n} $
\ and $\nu _{k}$ do not coincide (see equation~(\ref{E:distinct})). Let us
now analyze the behavior of the error $w_{\varepsilon }(0,x)$ if, for one
pair of numbers $n_{0}$ and~$k_{0}$, the eigenvalues do coincide (i.e. $%
\lambda _{n_{0}}=\nu _{k_{0}}$) but all the other pairs are still distinct.
Formulas (\ref{E:bound}) and (\ref{E:nice-estimate}) remain valid for all $%
k\neq k_{0}.$ For $k=k_{0}$ and $n=n_{0}$ equation~(\ref{E:integrals})
simplifies to%
\begin{align*}
I_{n_{0},k_{0}}(\varepsilon )& =\nu _{k_{0}}\int\limits_{0}^{1}\alpha
^{\prime }(\tau )\left[ 1-\cos \left( \frac{2\nu _{k_{0}}}{\varepsilon }\tau
\right) \right] d\tau -\frac{\varepsilon }{2}\int\limits_{0}^{1}\alpha
^{\prime \prime }(\tau )\sin \left( \frac{2\nu _{k_{0}}}{\varepsilon }\tau
\right) d\tau \\
& =\nu _{k_{0}}\int\limits_{0}^{1}\alpha ^{\prime }(\tau )d\tau =-\nu
_{k_{0}}.
\end{align*}%
Equation~(\ref{E:inprser}) for $k=k_{0}$ now becomes%
\begin{equation}
w_{k_{0}}(0)=\left( \frac{1}{\nu _{k_{0}}}\sum_{\substack{ n=1,  \\ n\neq
n_{0}}}^{\infty }u_{n}I_{n,k_{0}}(\varepsilon )<\psi _{k_{0}},\varphi
_{n}>_{c^{-2}}\right) -u_{n_{0}}<\psi _{k_{0}},\varphi _{n_{0}}>_{c^{-2}}.
\notag
\end{equation}%
It follows that, in addition to the error term converging to $0$ as $%
\varepsilon \rightarrow 0$ (shown in parentheses in the above formula), the
total error $w_{\varepsilon }(0,x)$ will contain an additional term equal to%
\begin{equation}
-u_{n_{0}}<\psi _{k_{0}},\varphi _{n_{0}}>_{c^{-2}}\psi _{k_{0}}(x).  \notag
\end{equation}%
Unless $<\psi _{k_{0}},\varphi _{n_{0}}>_{c^{-2}}$ happens to equal $0$
(simple examples show that this may or may not happen), the reconstruction
will have an error term that does not depend on $\varepsilon $ (or $T$), and
thus the gradual time reversal algorithm will not converge to $f.$

Clearly, if there are several pairs of eigenvalues $\left( \lambda
_{n_{j}},\nu _{k_{j}}\right) $, $j=1,..,J,$ ,\thinspace $J\leq \infty ,$
such that $\lambda _{n_{j}}=\nu _{k_{j}}$ and $<\psi _{k_{j}},\varphi
_{n_{j}}>_{c^{-2}}\neq 0,$ the reconstruction will contain a non-decaying
(with $\varepsilon $) error given by the following expression%
\begin{equation}
-\sum_{j=1}^{J}u_{n_{j}}<\psi _{k_{j}},\varphi _{n_{j}}>_{c^{-2}}\psi
_{k_{j}}(x).  \label{E:errorexpression}
\end{equation}%
The number $J$ of error terms in the above sum can happen to be infinite. In
this case (\ref{E:errorexpression}) is a converging series in $H^{1}(\Omega
),$ since in this space the error is bounded per Proposition 1.

\section{Particular cases\label{S:particular}}

Stronger convergence results can be obtained for simple domains where
eigenvalues of the Laplacians with proper boundary conditions are known.
Below we show that in the case of a circular cavity (in 2D) gradual time
reversal converges strongly in $H^{1}(\Omega )$ if $f(x)$ also belongs to
the latter space. We also analyze the case of a rectangular resonant cavity
with full and partial data (i.e. data measured on only one side of the
rectangle) and obtain somewhat unexpected results on weak convergence in
such cavities.

\subsection{Circular cavity}

Let us consider a particular case where the domain $\Omega $ is the unit
disk in $\mathbb{R}^{2}$ centered at the origin, with the boundary $\Sigma =%
\mathbb{S}^{1}.$ The data $U(t,z)$ are measured on all of $\Sigma $ (i.e., $%
\Sigma _{1}=\Sigma ),$ and the speed of sound is constant. Without loss of
generality we will (as we may) assume that $c(x)=1.$ Now the weighted
product $<\cdot ,\cdot >_{c^{-2}}$ and the norm $||\cdot ||_{c^{-2}}$
coincide with their standard counterparts $(\cdot ,\cdot )_{L_{2}}$ and $%
||\cdot ||_{2}.$Under the assumptions we made the error $w_{\varepsilon
}(t,\cdot )$ satisfies the zero Dirichlet boundary conditions on $\Sigma .$
The eigenfunctions $\psi _{k}$ and $\varphi _{n}$ are those of the Dirichlet
and Neumann Laplacians on the unit disk, respectively. They are normalized
in $L_{2}$, and in polar coordinates $(r,\theta )$ can be expressed using
double index notation as
\begin{align*}
\psi _{k,m}(r,\theta )& =D_{k,m}J_{|m|}(\nu _{k,|m|}r)e^{im\theta },\quad
k\in \mathbb{N},\quad m\in \mathbb{Z}, \\
\varphi _{n,l}(r,\theta )& =N_{n,l}J_{|l|}(\nu _{n,|l|}r)e^{il\theta },\quad
n\in \mathbb{N},\quad l\in \mathbb{Z},
\end{align*}%
where the eigenvalues $\nu _{k,|m|}$ and $\nu _{n,|l|}$ coincide with the
zeros $j_{k,|m|}$ and $j_{n,|l|}^{\prime }$ of the Bessel functions and
their derivatives%
\begin{align*}
\nu _{k,|m|}& =j_{k,|m|},\quad J_{|m|}(j_{k,|m|})=0,\quad k\in \mathbb{N}%
,\quad m\in \mathbb{Z}, \\
\lambda _{n,|l|}& =j_{n,|l|}^{\prime },\quad
J^\prime_{|l|}(j_{n,|l|})=0,\quad n\in \mathbb{N},\quad l\in \mathbb{Z},
\end{align*}%
and where the normalization constants $D_{k,m}$ and $N_{n,l}$ equal
\begin{align*}
D_{k,m}& =\left( 2\pi \int\limits_{0}^{1}J_{|m|}^{2}(\nu
_{k,|m|}r)rdr\right) ^{-\frac{1}{2}}=\frac{1}{\sqrt{\pi }|J_{m}^{\prime
}(\nu _{k,|m|})|}, \quad m\in \mathbb{Z},\quad k\in \mathbb{N}, \\
N_{n,l}& =\left( 2\pi \int\limits_{0}^{1}J_{|l|}^{2}(\lambda
_{n,|l|}r)rdr\right) ^{-\frac{1}{2}}=\left( \pi \left[ 1-\frac{l^{2}}{%
\lambda _{n,|l|}^{2}}\right] J_{|l|}^{2}(\lambda _{n,|l|})\right) ^{-\frac{1%
}{2}},
\end{align*}%
for all $l\in \mathbb{Z},$ $n\in \mathbb{N},$ except the case $(n,l)=(1,0)$,
when $N_{0,1}=1/\sqrt{\pi }.$

Now, the forward problem has a solution in the form%
\begin{equation}
u(t,r,\theta)=\sum_{l=-\infty}^{\infty}\sum_{n=1}^{\infty}u_{n,l}\varphi
_{n,l}(r,\theta)e^{il\theta}\cos(\lambda_{n,|l|}t),  \label{E:diskforward}
\end{equation}
where Fourier coefficients $u_{n,l}$ are related to the initial condition $%
u(0,r,\theta)=f(r,\theta)$ by%
\begin{equation}
u_{n,l}=f_{n,l,}\qquad f_{n,l}\equiv(f,\varphi_{n,l})_{L_{2}}.
\label{E:diskcoefs}
\end{equation}
Since $u(0,r,\theta)=f(r,\theta)\in H^{1}(\Omega),$%
\begin{equation}
||u(0,r,\theta)||_{H^{1}}^{2}=||f||_{H^{1}}^{2}=\sum_{l=-\infty}^{\infty}%
\sum_{n=1}^{\infty}|u_{n,l}|^{2}\left( 1+\lambda_{n,|l|}^{2}\right) <\infty.
\label{E:H1-disk}
\end{equation}

As before, $v_{\varepsilon }(t,r,\theta )$ represents solution of the
gradual time reversal problem, and $w_{\varepsilon }(0,r,\theta
)=v_{\varepsilon }(0,r,\theta )-f(r,\theta )$ represents the error of
approximating $f$ by $v_{\varepsilon }(0,r,\theta ).$ We expand $%
w_{\varepsilon }(0,r,\theta )$ in the series of $\psi _{k,m}(r,\theta )$%
\begin{align*}
w_{\varepsilon }(t,r,\theta )& =\sum_{m=-\infty }^{\infty
}\sum_{k=1}^{\infty }w_{m,k}(t)\psi _{k,m}(r,\theta ), \\
w_{m,k}(t)& =(w_{\varepsilon }(t,\cdot ,\cdot ),\psi _{k,m})_{L_{2}},
\end{align*}%
and apply the theoretical considerations of Section \ref{S:Theory}. Equation
(\ref{E:diffeq}) in the notation of the present section takes the form%
\begin{align}
w_{k,m}^{\prime \prime }(t)+\nu _{k,|m|}^{2}w_{k,m}(t)& =F_{k,m}(t),\qquad
F_{k,m}(t)\equiv (F_{\varepsilon },\psi _{k,m})_{L_{2}},
\label{E:newinverse} \\
w_{k,m}(T)& =0,\qquad w_{k,m}^{\prime }(T)=0,  \notag
\end{align}%
where
\begin{equation}
F_{\varepsilon }(t,r,\theta )\equiv -\varepsilon \left( 2\alpha ^{\prime
}(\varepsilon t)\frac{\partial u}{\partial t}(t,r,\theta )+\varepsilon
\alpha ^{\prime \prime }(\varepsilon t)u(t,r,\theta )\right) .  \notag
\end{equation}%
Differential equations~(\ref{E:newinverse}) are solved the same way as
before. Taking into account the orthogonality of the eigenfunctions $\psi
_{k,m}$ and $\varphi _{n,l}$ with \ $m\neq l,$ we thus obtain%
\begin{align}
\nu _{k,m}w_{k,m}(0)& =\sum_{n=1}^{\infty }u_{n,m}I_{n,k,m}(\varepsilon
)(\psi _{k,m},\varphi _{n,m})_{L_{2}},  \label{E:diskseries} \\
I_{n,k,m}(\varepsilon )& \equiv \varepsilon \int\limits_{0}^{1/\varepsilon }
\left[ 2\alpha ^{\prime }(\varepsilon t)\lambda _{n,|m|}\sin (\lambda
_{n,|m|}t)-\varepsilon \alpha ^{\prime \prime }(\varepsilon t)\cos (\lambda
_{n,|m|}t)\right] \sin (\nu _{k,|m|}t)dt,  \notag
\end{align}%
for $k,n\in \mathbb{N},$ $m\in \mathbb{Z}.$ Since for each fixed $m$ the
eigenvalues $\lambda _{n,|m|}$ and $\nu _{k,|m|}$ do not coincide,
conclusions of Section \ref{S:Weak} apply. Namely, from (\ref%
{E:good-estimate}) we obtain%
\begin{equation}
|\nu _{k,m}w_{k,m}(0)|\leq 2\varepsilon ^{M}B(M)\sum_{n=1}^{\infty
}|u_{n,m}|\,(\lambda _{n,|m|}+1)\frac{1}{|\lambda _{n,|m|}-\nu _{k,|m|}|^{M}}%
,\qquad m\in \mathbb{Z},  \label{E:diskbound1}
\end{equation}%
where the series converges for any $M\geq 3.$

It is well known that the zeros of the Bessel functions and their
derivatives becomes asymptotically equispaced for large $n$ and $k,$ and the
difference between the closest $\lambda_{n,m}$ and $\nu_{k,m}$ becomes close
to $\pi/2.$ One can prove a stronger statement (see Appendix): there exists
a constant $C_{5}$ such that the distance between $\lambda_{n,m}$ and $%
\nu_{k,m}$ is bounded uniformly in $m$ from below, namely
\begin{equation}
|\lambda_{n,m}-\nu_{k,m}|\geq C_{5}|2n-2k+1|,\qquad\forall m\in\mathbb{Z}.
\label{E:eigenbounds}
\end{equation}
Now (\ref{E:diskbound1}) can be re-written as%
\begin{equation}
|\nu_{k,m}w_{k,m}(0)|\leq\frac{2\varepsilon^{M}B(M)}{C_{5}^{M}}\sum
_{n=1}^{\infty}\frac{|u_{n,m}|\,(\lambda_{n,|m|}+1)}{|2n-2k+1|^{M}},\qquad
m\in\mathbb{Z}.  \label{E:simplerbound}
\end{equation}
Let us estimate the factor $|u_{n,m}|\,(\lambda_{n,|m|}+1)$ in the above
formula as follows%
\begin{equation}
|u_{n,m}|\,(\lambda_{n,|m|}+1)\leq\sqrt{\sum_{n=1}^{\infty}|u_{n,m}|^{2}\,(%
\lambda_{n,|m|}+1)^{2}}\leq\sqrt{2\sum_{n=1}^{\infty}|u_{n,m}|^{2}\,(%
\lambda_{n,|m|}^{2}+1)},  \notag
\end{equation}
substitute it in (\ref{E:simplerbound}) and take the root outside of the
summation sign thus obtaining%
\begin{equation}
\sum_{k=1}^{\infty}|\nu_{k,m}w_{k,m}(0)|^{2}\leq\left( \frac{4\varepsilon
^{M}B(M)}{C_{5}^{M}}\sum_{n=1}^{\infty}|u_{n,m}|^{2}\,(%
\lambda_{n,|m|}^{2}+1)\right) \sum_{n=1}^{\infty}\frac{1}{|2n-2k+1|^{M}}.
\label{E:evensimplerbound}
\end{equation}
The last series in the above equation is convergent and can be uniformly
bounded:%
\begin{equation}
\sum_{n=1}^{\infty}\frac{1}{|2n-2k+1|^{M}}\leq\sum_{n=-\infty}^{\infty}\frac{%
1}{|2n-2k+1|^{M}}=\sum_{l=-\infty}^{\infty}\frac{1}{|2l+1|^{M}}\equiv C_{6}.
\notag
\end{equation}
This allows us to simplify (\ref{E:evensimplerbound}) further:%
\begin{align}
\sum_{k=1}^{\infty}|\nu_{k,m}w_{k,m}(0)|^{2} & \leq C_{7}(M)\varepsilon
^{M}\sum_{n=1}^{\infty}|u_{n,m}|^{2}\,(\lambda_{n,|m|}^{2}+1),
\label{E:newestimate} \\
\text{where }C_{7}(M) & \equiv\frac{4B(M)}{C_{5}^{M}}\sum_{l=-\infty
}^{\infty}\frac{1}{|2l+1|^{M}}.  \notag
\end{align}
Finally, we can find a bound on $||w_{\varepsilon}(0,r,\theta)||_{H^{1}.}$.
First, we notice that%
\begin{equation}
||w_{\varepsilon}||_{H^{1}}^{2}=||w_{\varepsilon}||_{2}^{2}+||\nabla
w_{\varepsilon}||_{2}^{2}=\sum_{\substack{ m\in\mathbb{Z}  \\ k\in\mathbb{N}
}}\left( |w_{k,m}|^{2}+|\nu_{k,m}w_{k,m}|^{2}\right) \leq\left( \frac {1}{%
\nu_{1,0}^{2}}+1\right) \sum_{\substack{ m\in\mathbb{Z}  \\ k\in\mathbb{N}}}%
|\nu_{k,m}w_{k,m}|^{2},  \notag
\end{equation}
for any $t\in\lbrack0,\infty).$ By combining this inequality\ (taken at $%
t=0) $ with (\ref{E:newestimate}) and taking into account (\ref{E:H1-disk})
we obtain%
\begin{align*}
||w_{\varepsilon}(0,r,\theta)||_{H^{1}}^{2} & \leq\left( \frac{1}{\nu
_{1,0}^{2}}+1\right) \sum_{m\in\mathbb{Z}}\sum_{k\in\mathbb{N}}|\nu
_{k,m}w_{k,m}(0)|^{2} \\
& =C_{7}(M)\varepsilon^{M}\left( \frac{1}{\nu_{1,0}^{2}}+1\right) \sum_{m\in%
\mathbb{Z}}\sum_{n\in\mathbb{N}}|u_{n,m}|^{2}\,\left(
\lambda_{n,|m|}^{2}+1\right) , \\
& =C_{7}(M)\varepsilon^{M}\left( \frac{1}{\nu_{1,0}^{2}}+1\right)
||f||_{H^{1}}^{2}\underset{\varepsilon\rightarrow0}{\rightarrow}0.
\end{align*}
Thus, we have proven the following

\begin{theorem}
If $\Omega$ is the unit disk in $\mathbb{R}^{2},$ the data are measured on
the whole circle (i.e., $\Sigma_{1}=\mathbb{S}^{1}),$ the cut-off function $%
\alpha(t)$ satisfies the assumptions made in Section \ref{S:Theory}, and the
initial pressure $f\in H^{1}(\Omega),$ then the result $v_{\varepsilon
}(0,\cdot)$ of gradual time reversal converges to $f$ in $H^{1}(\Omega)$
(strongly), as $\varepsilon\rightarrow0$ (or as $T\rightarrow\infty).$
\end{theorem}

\begin{figure}[t]
\begin{center}
\subfigure[Phantom]{
\includegraphics[width=1.75in,height=1.75in]{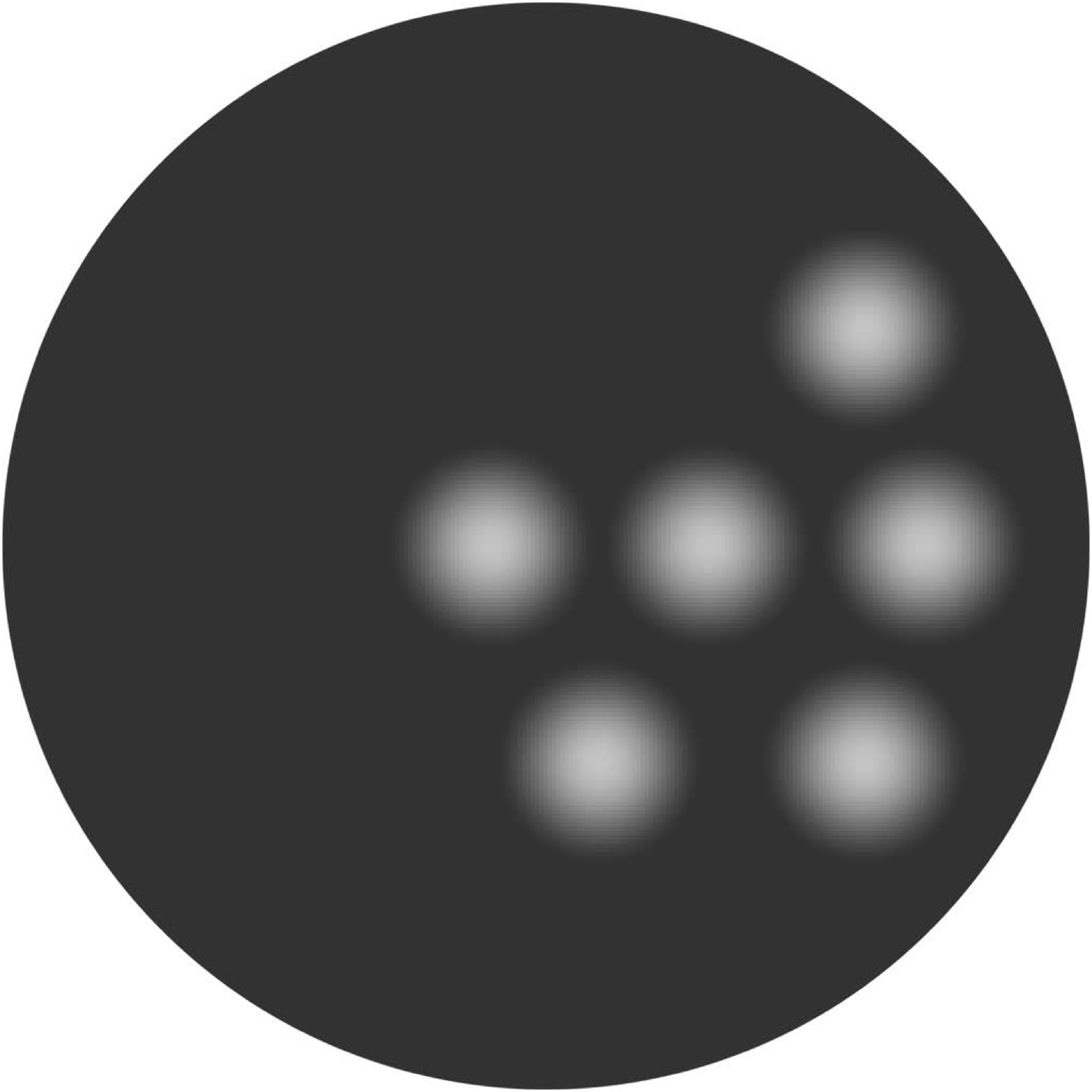}}
\subfigure[Reconstruction, 5.3 sec.]{
\includegraphics[width=1.75in,height=1.75in]{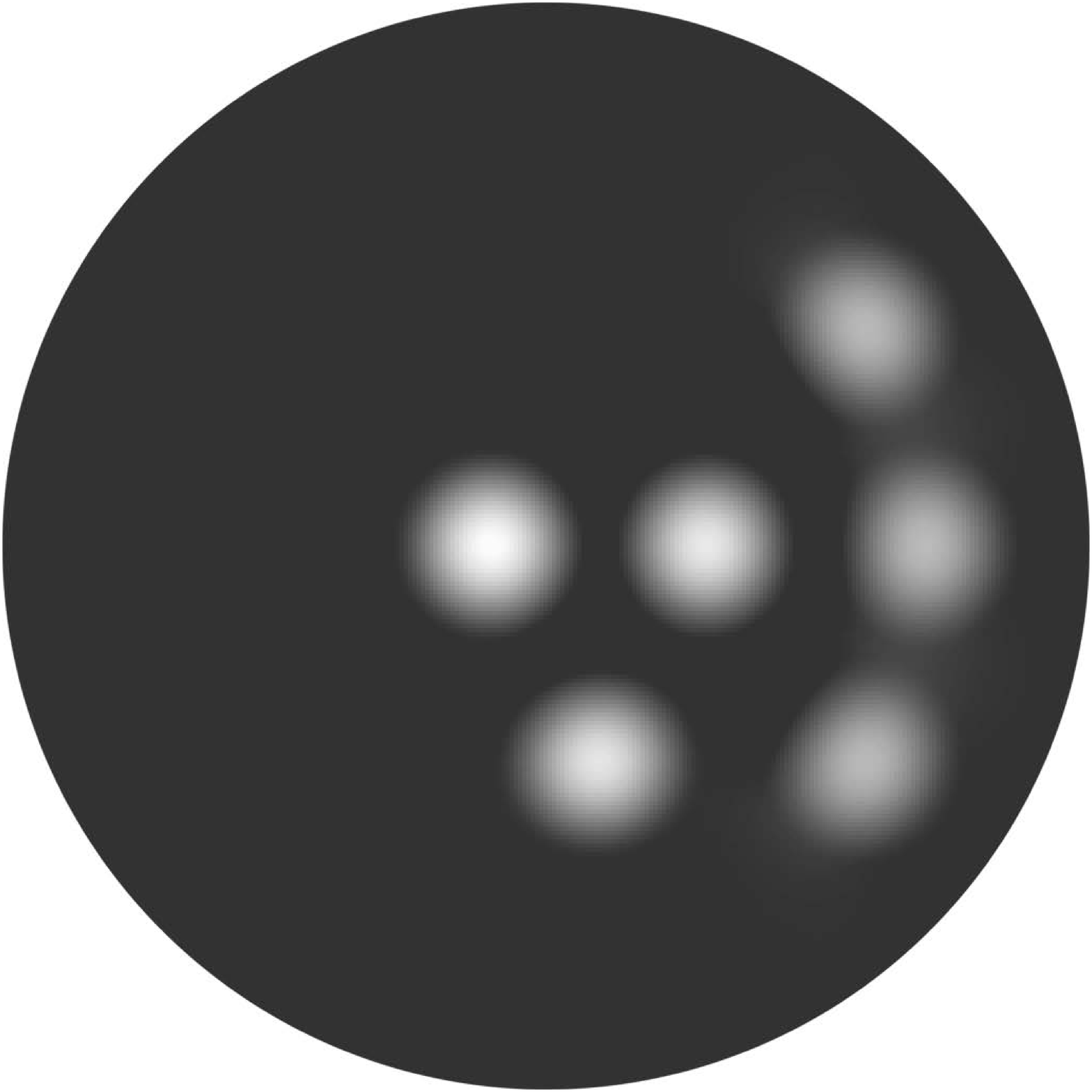}}
\subfigure[Reconstruction, 10.6 sec.]{
\includegraphics[width=1.75in,height=1.75in]{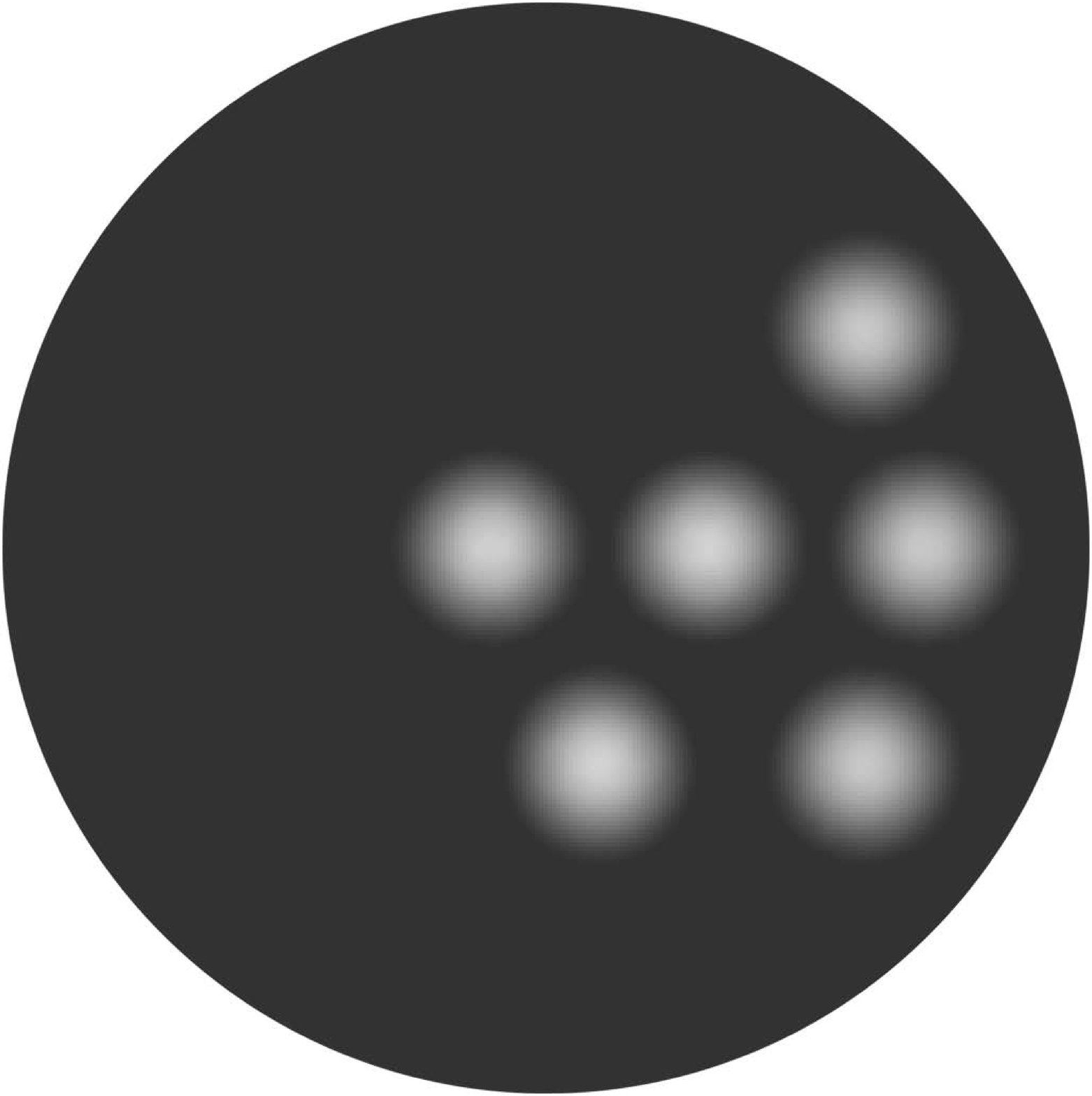}}
\end{center}
\caption{Reconstructions in the unit disk with model measurement times 5.3
and 10.6 seconds, correspondingly}
\end{figure}

\subsubsection{Numerical example}

The following numerical example illustrates how gradual time reversal works
in a circular domain. In our simulation the phantom was modeled by a sum of
three rotationally symmetric smooth functions defined on the unit disk, as
shown in Figure~1(a). The measurements were simulated by tabulating the
series solution (equations~(\ref{E:diskforward}) and (\ref{E:diskcoefs})) at
1024 equispaced points on the unit circle that represent detectors. Two
reconstructions were computed using gradual time reversal with $T=5.3$ sec.
and $T=10.6$ sec. of model time. (For comparison, $2$ sec. of model time is
the time needed for a wave to propagate once along the diameter of the
disk.) These computations were performed using an efficient algorithm~\cite%
{Ben} for solving the wave equation on a reduced polar grid in a circular
domain. The reconstructed images are shown in Figure~1(a) and (b),
correspondingly. While the reconstruction with $T=5.3$ sec. is not quite
accurate, the image corresponding to $T=10.6$ sec. looks quite close to the
original in a gray scale picture. Simulations with larger $T$ (not shown
here) yield images that are very close to the phantom not only visually but
quantitatively as well.

\subsection{Rectangular cavity}

A rectangular resonant cavity arises naturally when the object is surrounded
by flat detector assemblies that act as reflectors (see, for example,~\cite%
{Cox}). A very fast Fourier-based reconstruction algorithm has been
developed by the authors jointly with B.T. Cox for such a configuration.
However, gradual time reversal using finite differences on a Cartesian grid
is a much simpler (although slower) method, and, unlike the former
algorithm, it can be easily implemented for a variable (but known) speed of
sound. In addition, the simplicity of a rectangular domain allows us to
illustrate several interesting properties of gradual time reversal.

For simplicity, we present below the 2D case; extension to the 3D is
straightforward.

\subsubsection{Time reversal with full data}

\paragraph{Square domain}

\begin{figure}[t]
\begin{center}
\subfigure[Phantom]{
\includegraphics[width=1.75in,height=1.75in]{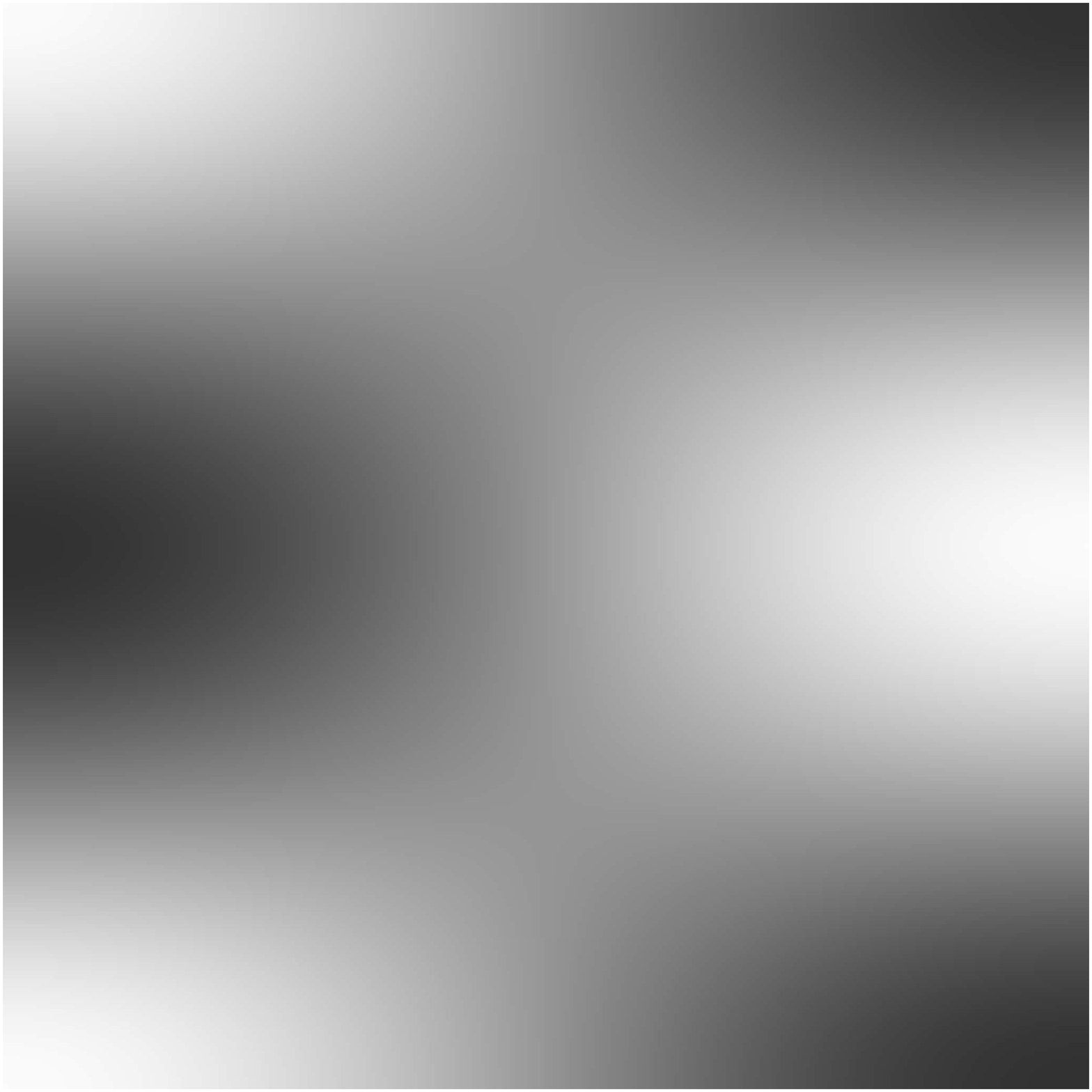}}
\subfigure[Reconstruction, 100 sec.]{
\includegraphics[width=1.75in,height=1.75in]{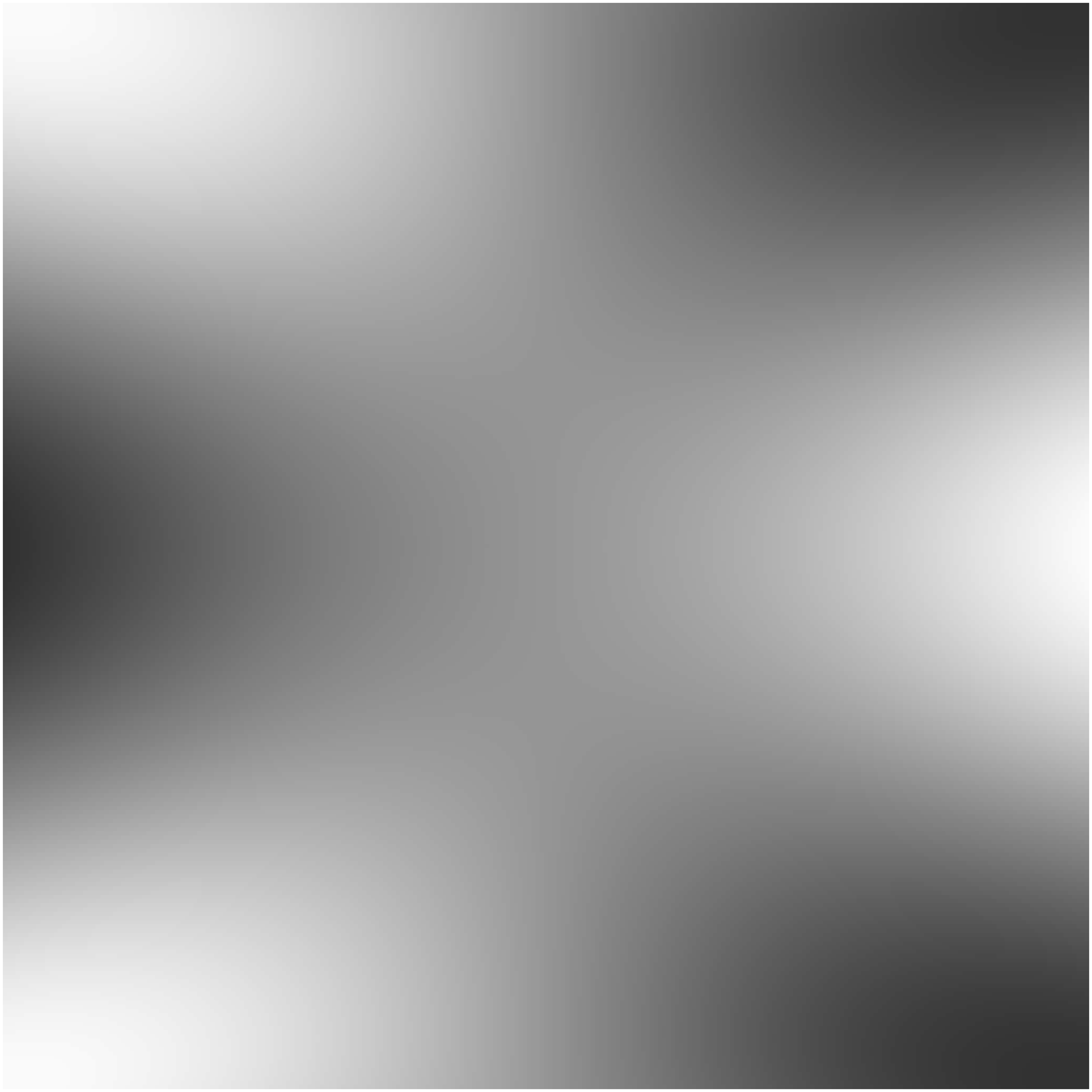}}
\subfigure[Difference image]{
\includegraphics[width=1.75in,height=1.75in]{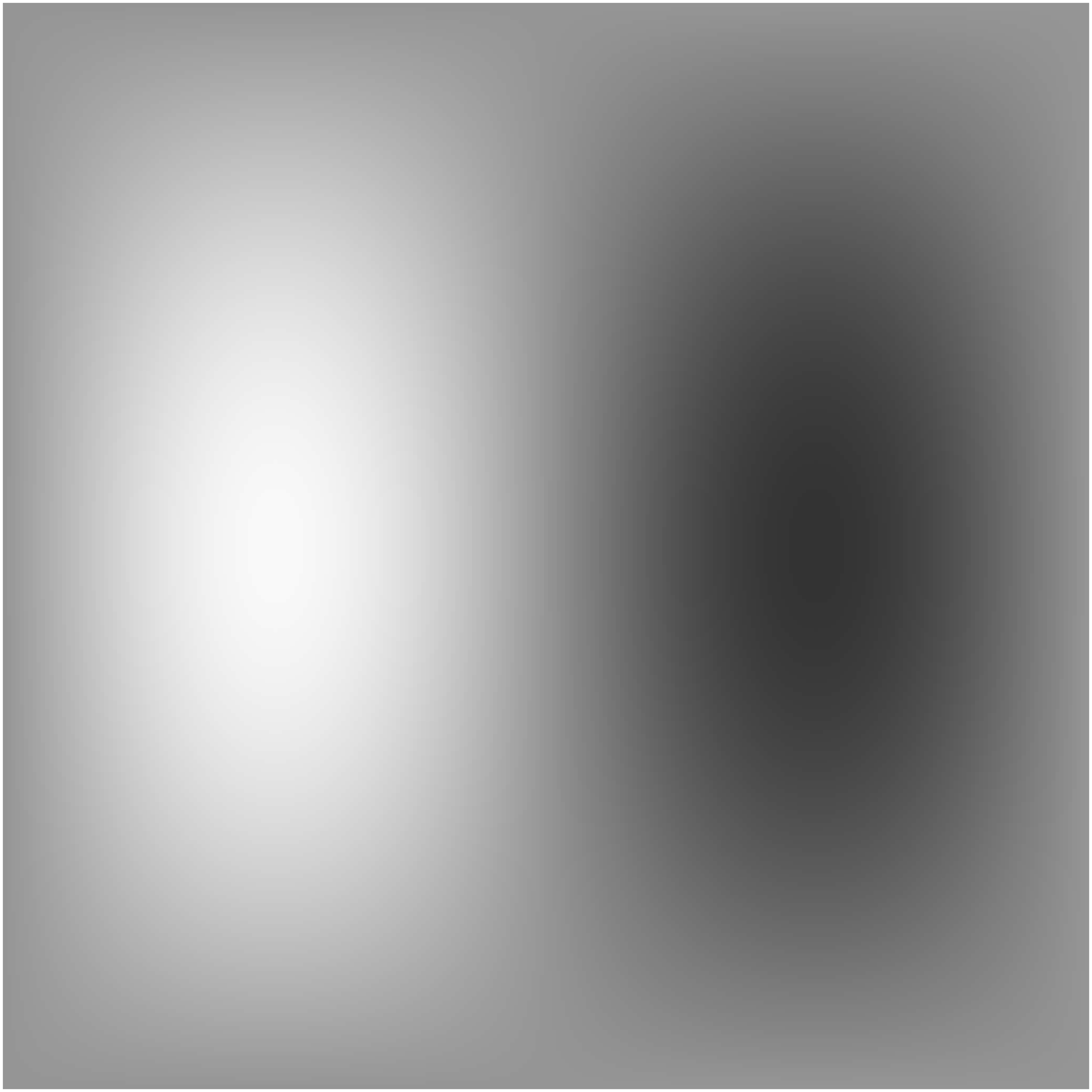}}
\end{center}
\caption{Reconstruction of a phantom $\frac{2}{\protect\pi}\cos
x_{1}\cos2x_{2}$ in a square, showing incomplete convergence. The difference
between the phantom and the reconstruction (shown on a different gray scale)
is very close to $\frac{64}{9\protect\pi^{3}}\sin2x_{1}\sin x_{2}$ }
\end{figure}
First, let's assume that the domain $\Omega$ is a square $(0,\pi
)\times(0,\pi)$ and that the data are measured on the whole boundary, i.e. $%
\Sigma_{1}=\Sigma.$ It is convenient to number the eigenfunctions using
double indices. In particular, the eigenfunctions $\varphi_{n,l}$ and
eigenvalues $\lambda_{n,l}$ needed for solving the forward problem are those
of the Neumann Laplacian on $\Omega$:%
\begin{equation}
\varphi_{n,l}(x)=N_{n,l}\cos nx_{1}\cos
lx_{2},\quad\lambda_{n,l}^{2}=n^{2}+l^{2},\quad n,l=0,1,2,3,...,
\label{E:Neumannsquare}
\end{equation}
with normalization constants $N_{n,l}=\frac{2}{\pi}$ if $n,l>0,$ $N_{0,0}=%
\frac{1}{\pi},$ $N_{0,1}=N_{1,0}=\frac{\sqrt{2}}{\pi}.$ The eigenfunctions $%
\psi_{k,m}$ and eigenvalues $\nu_{k,m}$ arising in the analysis of time
reversal are those of the Dirichlet Laplacian on $\Omega$:%
\begin{equation}
\psi_{k,m}(x)=\frac{2}{\pi}\sin kx_{1}\sin
mx_{2},\quad\gamma_{k,m}^{2}=k^{2}+m^{2},\quad k,m=0,1,2,3,...  \notag
\end{equation}
We notice that in the cases $(k,m)=(n,l)$ or $(m,k)=(n,l)$ the eigenvalues
coincide; corresponding eigenfunctions are orthogonal in the former case but
not always in the latter case:%
\begin{align*}
\lambda_{k,m} & =\lambda_{m,k}=\gamma_{k,m}=\gamma_{m,k}, \\
(\psi_{k,m},\varphi_{k,m})_{L_{2}} & =0,\quad k,m=0,1,2,3,... \\
(\psi_{m,k},\varphi_{k,m})_{L_{2}} & =0\quad\text{if }k-m\text{ is even,} \\
(\psi_{m,k},\varphi_{k,m})_{L_{2}} & \neq0\quad\text{if }k-m\text{ is odd.}
\end{align*}
In addition, there are some other pairs of coinciding eigenvalues with
non-orthogonal eigenfunctions (e.g. $\lambda_{8,1} = \lambda_{7,4} =
\gamma_{8,1} = \gamma_{7,4}$). This implies that the result of the gradual
time reversal will not converge to $f(x)$. The residual error will be given
by an expression with infinite number of terms similar to (\ref%
{E:errorexpression}) (with modifications needed to account for the double
indexing of eigenfunctions).

For a simple example, consider initial conditions $f(x)=\varphi_{1,2}(x)=%
\frac{2}{\pi}\cos x_{1}\cos2x_{2}.$ Then all coefficients $u_{n,l}$ (except $%
u_{1,2}=1$) are equal to 0. This implies that gradual time reversal will
converge to $f(x)+\mathrm{Err}(x)$ with
\begin{equation}
\mathrm{Err}(x)=-(\psi_{2,1},\varphi_{1,2})_{L_{2}}\psi_{2,1}(x)=\frac{32}{9\pi^{2}}%
\psi_{2,1}(x)=\frac{64}{9\pi^{3}}\sin2x_{1}\sin x_{2}.  \label{E:errorsquare}
\end{equation}
Figure~2 presents the results of numerical simulation we ran to further
illustrate this situation. The model time in this example was about $314$
sec. which corresponds to a hundred bounces of a wave between the opposite
sides of a square. \ The phantom $f(x)$ is shown in part (a) of the figure,
and part (b) shows the reconstruction. Figure~2(c) presents the error in the
reconstruction (shown on a different gray scale); it turns out to be very
close to the theoretical prediction $\mathrm{Err}(x)$ given by equation~(\ref%
{E:errorsquare}).

\paragraph{Rectangle with incommensurable sides}

Let us now consider a rectangular domain $(0,A)\times(0,B).$ Let us assume
that $A$ and $B$ are incommensurable numbers, for example $A$ is rational
and $B$ is irrational. Now
\begin{align*}
\varphi_{n,l}(x) & =N_{n,l}\cos\frac{\pi nx_{1}}{A}\cos\frac{\pi nx_{2}}{B}%
,\quad\lambda_{n,l}^{2}=\frac{\pi^{2}n^{2}}{A^{2}}+\frac{\pi^{2}l^{2}}{B^{2}}%
,\quad n,l=0,1,2,3,..., \\
\psi_{k,m}(x) & =\frac{2\pi}{\sqrt{AB}}\sin\frac{\pi kx_{1}}{A}\sin\frac{\pi
mx_{2}}{B},\quad\gamma_{k,m}^{2}=\frac{\pi^{2}k^{2}}{A^{2}}+\frac{%
\pi^{2}m^{2}}{B^{2}},\quad k,m=0,1,2,3,...,
\end{align*}
where $N_{n,l}=\sqrt{(\varphi_{n,l},\varphi_{n,l})_{L_{2}}}$ are the
normalization constants. The only situations when values of $\lambda_{n,l}$
and $\gamma_{k,m}$ coincide is when $(k,m)=(n,l).$ However, it is easy to
check by direct computation that in this case $(\psi_{n,l},\varphi
_{n,l})_{L_{2}}=0$ and the error terms in the form (\ref{E:errorexpression})
vanish. Therefore, according to the analysis of Section \ref{S:Theory}, the
result of the gradual time reversal will converge weakly (in $H^{1}(\Omega))$
to $f(x)$ as $T\rightarrow\infty.$

\subsubsection{Time reversal with partial data}

\begin{figure}[t]
\begin{center}
\subfigure[Phantom]{
\includegraphics[width=1.75in,height=1.75in]{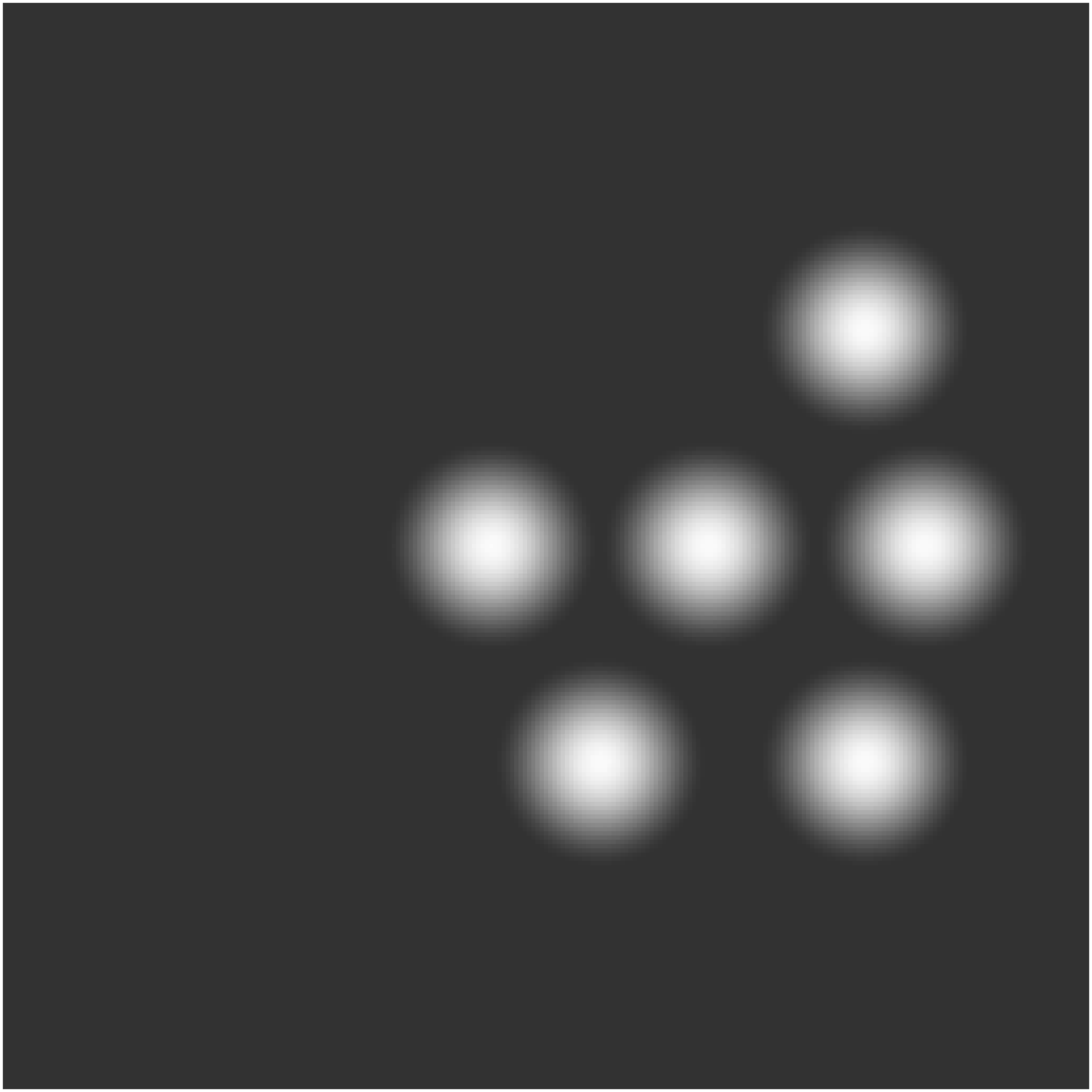}}
\subfigure[Reconstruction, 100 sec.]{
\includegraphics[width=1.75in,height=1.75in]{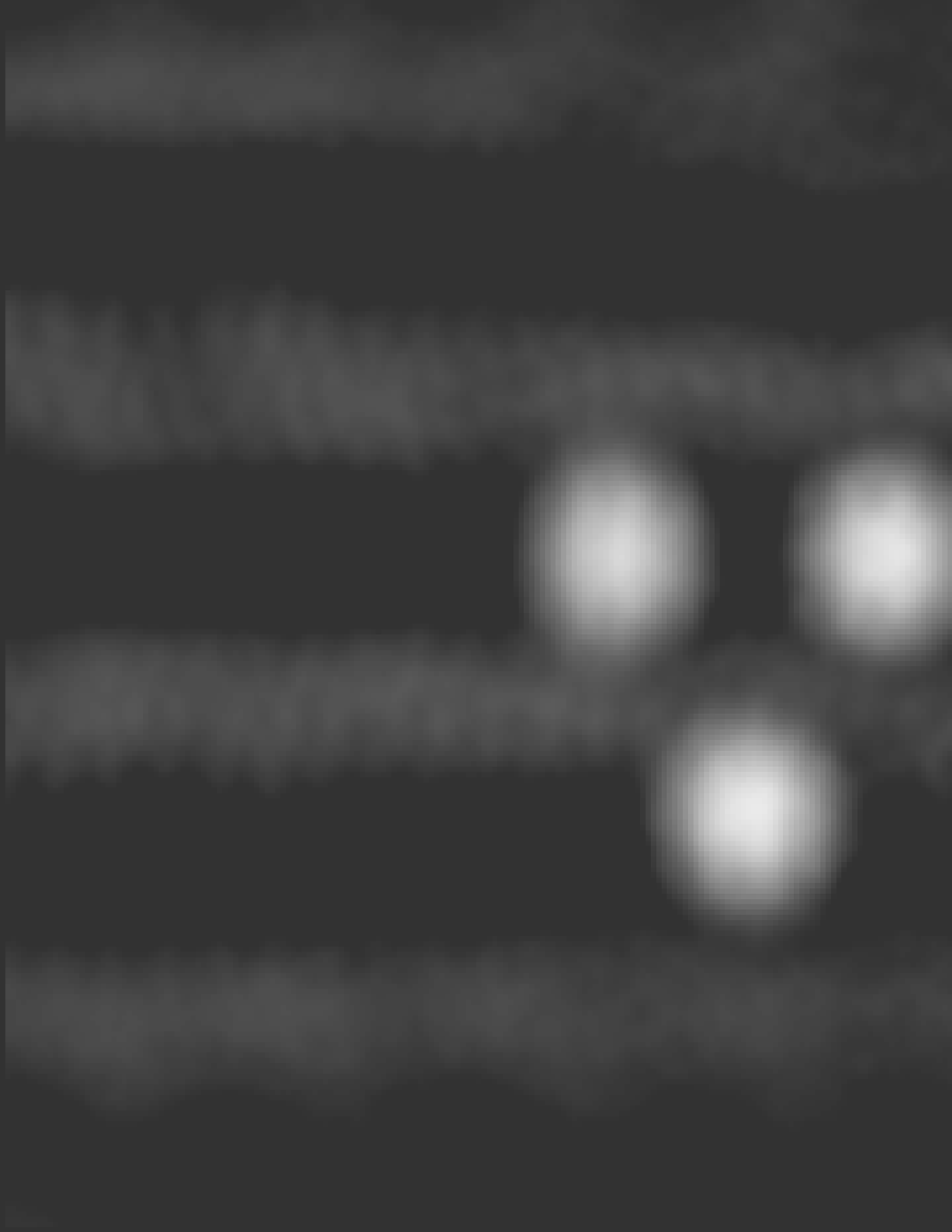}}
\subfigure[Difference image]{
\includegraphics[width=1.75in,height=1.75in]{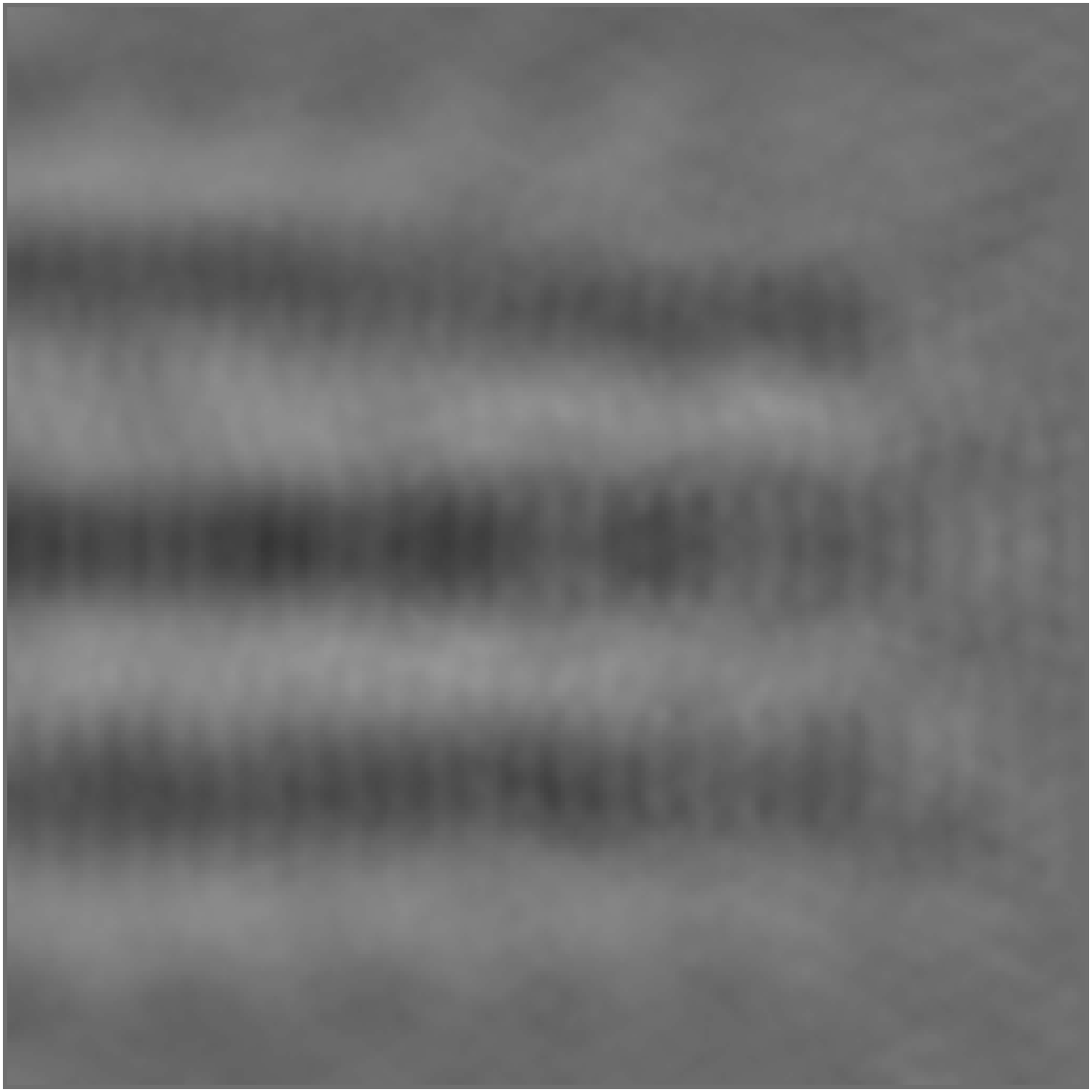}}
\end{center}
\caption{Reconstruction in a square (a) phantom (b) reconstruction from the
data measured on the right side, $T=490$ sec. (c) error shown on a shifted
gray scale}
\end{figure}

We return to the case of the square domain $\Omega =$ $(0,\pi )\times (0,\pi
)$, \ but this time assume that the measurements are made on only one side
of the square, corresponding to $x_{1}=\pi $ (this side plays the role of $\
\Sigma _{1}).$ In this case the eigenfunctions $\varphi _{n,l}(x)$ needed to
solve the forward problem are still given by equation~(\ref{E:Neumannsquare}%
). To analyze gradual time reversal one needs the eigenfunctions $\psi
_{k,m}(x)$ of the Laplacian satisfying Dirichlet boundary condition on the
right side of the square and Neumann conditions on the three other sides:
\begin{equation}
\psi _{k,m}(x)=\frac{2}{\pi }\cos \left( \left( k+\frac{1}{2}\right)
x_{1}\right) \cos mx_{2},\quad \gamma _{k,m}^{2}=\left( k+\frac{1}{2}\right)
^{2}+m^{2},\quad k,m=0,1,2,3,...  \notag
\end{equation}%
It follows immediately that, since $\lambda _{n,l}^{2}=n^{2}+l^{2}$ is an
integer number and $\gamma _{k,m}^{2}$ is not, $\lambda _{n,l}$ and $\gamma
_{k,m}$ never coincide, and therefore the result of gradual time reversal
converges weakly in $H^{1}(\Omega )$ to the sought initial condition $f(x)$
as $T\rightarrow \infty .$ It may seem counter-intuitive that time reversal
with partial data yields weak convergence while time reversal with full data
does not converge. However, this just means that the latter application of
this technique is flawed and does not extract full information from the data.

A further look at the eigenfunctions $\psi _{k,m}(x)$ and $\varphi _{n,l}(x)$
reveals that they are orthogonal if $m\neq l,$and, therefore, series
representing the error in gradual time reversal will partially decouple
similarly to those arising in the analysis of the circular domain. \ In
general, the rate of convergence of individual modes depends on the
differences of eigenvalues (see, for example equations~(\ref{E:converg}) and
(\ref{E:C4})). If some of these differences are small than the corresponding
constant $C_{4}(M,k)$ is large and the convergence will be slow. In the case
of the square domain, due to partial orthogonality of eigenmodes, we only
need to consider the following differences%
\begin{equation}
|\lambda _{n,m}-\gamma _{k,m}|=\left\vert \sqrt{n^{2}+m^{2}}-\sqrt{\left( k+%
\frac{1}{2}\right) ^{2}+m^{2}}\right\vert .  \notag
\end{equation}%
For large values of $m$ and for $x$ of order of 1 the following Taylor
expansion holds%
\begin{equation}
\sqrt{m^{2}+x}=m\sqrt{1+x/m^{2}}\thickapprox m+\frac{x}{2m}.  \notag
\end{equation}%
Therefore, if $n$ and $k$ are of order of 1 and $m$ is large%
\begin{equation}
|\lambda _{n,m}-\gamma _{k,m}|\thickapprox \frac{n^{2}-\left( k+\frac{1}{2}%
\right) ^{2}}{2m};  \notag
\end{equation}%
if we let $m$ grow to $\infty $ with fixed $n$ and $k$, the quantity $%
|\lambda _{n,m}-\gamma _{k,m}|$ converges to 0, which means that convergence
of the corresponding modes is getting slower. This implies that even for
large values of $T$, the error will contain components with large $m$ and
small $k$ manifesting themsselves as waves propagating in the near-vertical
direction.

The following numerical example illustrates this situation. Figure~3(a)
presents the phantom (the same as in the disk simulation), Figure~3(b) shows
the reconstruction corresponding to $T=490$ sec. (or to approximately 156
bounces of a wave between the opposite sides of the square). Figure~3(c)
demonstrates the reconstruction error on a shifted gray scale. One can see
that in spite of a large value of $T$ in the simulation the error remains
noticeable; it consists mostly of the waves with vertical wavefronts as
predicted by the above analysis.

This suggests that, if the data is measured on two perpendicular sides of
the square, one could use the following reconstruction algorithm: run
gradual time reversal separately for each side, filter the two so-obtained
images to remove the waves propagating parallel to the acquisition sides,
and then add resulting images together. This, indeed, can be done; however,
since a fast and rigorously proven method~\cite{Cox} is available for such
an acquisition scheme, we will not elaborate further on this topic.

\section*{Acknowledgements}

The authors would like to thank G.~Berkolaiko, L.~Friedlander, L.~Hermi, and
P.~Kuchment for helpful discussions. The second author gratefully
acknowledges support by the NSF, through award NSF/DMS-1211521.

\section*{Appendix}

In the case of a unit disk domain the eigenvalues $\nu_{m,k}$ and $%
\lambda_{m,k}$ of the Dirichlet and Neumann Laplacians coincide,
correspondingly, \ with the positive roots $j_{m,k}$ and $j_{m,k}^{\prime}$
of the Bessel functions $J_{m}(x)$ and its derivative $J_{m}^{\prime}(x)$
(with the exception that $x=0$ is counted as the first zero of $%
J_{0}^{\prime}(x)$ and not counted for $J_{m}^{\prime}(x)$ with $m>0)$. In
this section we establish the fact that the distance between these roots is
uniformly bounded from below (in the sense of equation~(\ref{E:eigenbounds}%
)).

It is well known that the roots of $J_{m}(x)$ and $J_{m}^{\prime}(x)$
interlace \cite{Watson,Abramowitz}:%
\begin{equation}
j_{m,k}<j_{m,k+1}^{\prime}<j_{m,k+1},\qquad k=1,2,3,...  \notag
\end{equation}
It is also known that for $m>0$ all non-zero roots are greater than $m$;
more precisely~\cite{Watson}:%
\begin{equation}
j_{m,1}>j_{m,1}^{\prime}>\sqrt{m(m+2)},\qquad m>0,  \notag
\end{equation}
and that asymptotically (for large $k$ and fixed $m)$ these roots become
equispaced \cite{Abramowitz}:
\begin{align}
j_{m,k} & \sim\left( k+\frac{1}{2}m-\frac{1}{4}\right) \pi ,
\label{E:asympt1} \\
j_{m,k}^{\prime} & \sim\left( k+\frac{1}{2}m-\frac{3}{4}\right) \pi.
\label{E:asympt2}
\end{align}
Therefore, as $k\rightarrow\infty$%
\begin{align*}
j_{m,k}-j_{m,k}^{\prime} & \rightarrow\frac{\pi}{2}, \\
j_{m,k+1}^{\prime}-j_{m,k} & \rightarrow\frac{\pi}{2}.
\end{align*}
We, however, need a uniform lower bound on $|j_{m,k}-j_{m,l}^{\prime}|$
valid for all values of $m,$ $k,$ and $l.$ Such a bound on the distance
between the roots of $J_{m}(x)$ is known \cite{Elbert}:
\begin{equation}
|j_{m,k}-j_{m,l}|>\pi|k-l|,\qquad m>1/2.  \label{E:Elbert}
\end{equation}
In addition, we need the following result which we have not found in the
literature:

\begin{lemma}
For any $m\geq1$ ($m$ does not have to be integer) the distance between the
adjacent roots of $J_{m}^{\prime}$ and $J_{m}$ is uniformly bounded:
\begin{align}
j_{m,k}-j_{m,k}^{\prime} & \geq\sqrt{2},\qquad k=2,3,4,...  \label{E:lemma1}
\\
j_{m,k+1}^{\prime}-j_{m,k} & \geq1,\qquad k=1,2,3,...  \label{E:lemma2}
\end{align}
\end{lemma}

\begin{proof}
Let us consider the open interval $I_{k}=(j_{m,k},j_{m,k+1})$ between the
adjacent zeros of $J_{m}.$ Let us first assume that $J_{m}$ is positive on $%
I_{k}.$ Then $J_{m}$ attains its local maximum $\max\limits_{I_{k}}J_{m}(x)$
at the point $j_{m,k+1}^{\prime }\in I_{k}.$ Recall that $J_{m}$ satisfies
the Bessel equation
\begin{equation}
x^{2}J_{m}^{\prime \prime }(x)+xJ_{m}^{\prime }(x)+(x^{2}-m^{2})J_{m}(x)=0,
\notag
\end{equation}%
which for $x>0$ can be re-written in the following form:%
\begin{equation}
\left( xJ_{m}^{\prime }(x)\right) ^{\prime }=-\frac{x^{2}-m^{2}}{x}J_{m}(x).
\label{E:Bessel-eq}
\end{equation}%
By integrating the latter equation from $j_{m,k+1}^{\prime }$ to $x$ we
obtain%
\begin{equation}
J_{m}^{\prime }(x)=-\frac{1}{x}\int\limits_{j_{m,k+1}^{\prime }}^{x}\frac{%
t^{2}-m^{2}}{t}J_{m}(t)dt.  \label{E:antider-1}
\end{equation}%
Now let us use equation~(\ref{E:antider-1}) and integrate $-J_{m}^{\prime
}(x)$ from $j_{m,k+1}^{\prime }$ to $j_{m,k+1},$ taking into account that $%
J_{m}(j_{m,k+1})=0,$ that $\max_{I_{k}}J_{m}(x)=J_{m}(j_{m,k+1}^{\prime }),$
and that $t>m$ on $(j_{m,k+1}^{\prime },j_{m,k+1})$:%
\begin{equation}
\max_{I_{k}}J_{m}(x)=\int\limits_{j_{m,k+1}^{\prime }}^{j_{m,k+1}}\frac{1}{x}%
\left[ \int\limits_{j_{m,k+1}^{\prime }}^{x}\frac{t^{2}-m^{2}}{t}J_{m}(t)dt%
\right] dx\leq \max_{I_{k}}J_{m}(x)\int\limits_{j_{m,k+1}^{\prime
}}^{j_{m,k+1}}\frac{1}{x}\left[ \int\limits_{j_{m,k+1}^{\prime }}^{x}tdt%
\right] dx.  \notag
\end{equation}%
Dividing both sides of the above inequality by $\max\limits_{I_{k}}J_{m}(x)$
yields:
\begin{align*}
1& \leq \int\limits_{j_{m,k+1}^{\prime }}^{j_{m,k+1}}\frac{1}{x}\left[
\int\limits_{j_{m,k+1}^{\prime }}^{x}tdt\right] dx=\int\limits_{j_{m,k+1}^{%
\prime }}^{j_{m,k+1}}\frac{1}{2x}(x-j_{m,k+1}^{\prime })(x+j_{m,k+1}^{\prime
})dx \\
& \leq \int\limits_{j_{m,k+1}^{\prime }}^{j_{m,k+1}}(x-j_{m,k+1}^{\prime
})dx=\frac{(j_{m,k+1}-j_{m,k+1}^{\prime })^{2}}{2}.
\end{align*}%
Therefore%
\begin{equation}
j_{m,k+1}-j_{m,k+1}^{\prime }\geq \sqrt{2}.  \notag
\end{equation}%
Similarly, in order to bound $j_{m,k+1}^{\prime }-j_{m,k},$ integrate
equation~(\ref{E:antider-1}) from $j_{m,k}$ to $j_{m,k+1}^{\prime }$:%
\begin{equation}
\max_{I_{k}}J_{m}(x)=\int\limits_{j_{m,k}}^{j_{m,k+1}^{\prime }}\frac{1}{x}%
\left[ \int\limits_{x}^{j_{m,k+1}^{\prime }}\frac{t^{2}-m^{2}}{t}J_{m}(t)dt%
\right] dx\leq \max_{I_{k}}J_{m}(x)\int\limits_{j_{m,k}}^{j_{m,k+1}^{\prime
}}\frac{1}{x}\left[ \int\limits_{x}^{j_{m,k+1}^{\prime }}\frac{t^{2}-m^{2}}{t%
}dt\right] dx.  \notag
\end{equation}%
By dividing both sides by $\max\limits_{I_{k}}J_{m}(x)$ we obtain:%
\begin{align*}
1 &\leq \int\limits_{j_{m,k}}^{j_{m,k+1}^{\prime}}\frac{1}{x}\left[
\int\limits_{x}^{j_{m,k+1}^{\prime}}\frac{t^{2}-m^{2}}{t}dt\right] dx\leq
\int\limits_{j_{m,k}}^{j_{m,k+1}^{\prime }}\frac{2}{x}\left[
\int\limits_{x}^{j_{m,k+1}^{\prime }}(t-m)dt\right] dx \\
&\leq \int\limits_{j_{m,k}}^{j_{m,k+1}^{\prime }}\frac{1}{x}\left[
j_{m,k+1}^{\prime }-x^{2}-2m(j_{m,k+1}^{\prime }-x)\right] dx\leq
\int\limits_{j_{m,k}}^{j_{m,k+1}^{\prime }}(j_{m,k+1}^{\prime }-x)\frac{%
j_{m,k+1}^{\prime }+x-2m}{j_{m,k}}dx
\end{align*}%
\begin{align*}
& \leq \frac{2(j_{m,k+1}^{\prime }-m)}{j_{m,k}}\int%
\limits_{j_{m,k}}^{j_{m,k+1}^{\prime }}(j_{m,k+1}^{\prime
}-x)dx=(j_{m,k+1}^{\prime }-j_{m,k})^{2}\frac{j_{m,k+1}^{\prime }-m}{j_{m,k}}
\\
& \leq (j_{m,k+1}^{\prime }-j_{m,k})^{2}\frac{j_{m,k+1}^{\prime }-1}{j_{m,k}}%
,
\end{align*}%
which implies $j_{m,k+1}^{\prime }-j_{m,k}\geq 1.$ This proves the Lemma for
all intervals $I_{k}$ on which $J_{m}$ is positive. In order to prove it for
the intervals where $J_{m}$ is negative, replace $J_{m}$ by $-J_{m}$ and
repeat the proof above.
\end{proof}

\begin{proposition}
There is constant $C>0$ such that for any integer numbers $m\geq0,$ $k\geq1,$
and $l\geq1$ the distance between the roots $j_{m,l}$ of $J_{m}$ and $%
j_{m,k}^{\prime}$ of $J_{m}^{\prime}$ is bounded from below:%
\begin{equation}
\left\vert j_{m,k}-j_{m,l}^{\prime}\right\vert \geq C\left\vert
2k-2l+1\right\vert  \label{E:thm}
\end{equation}
\end{proposition}

\begin{proof}
First consider any function $J_{m}$ for $m\geq 1.$ If $l\leq k,$ due to (\ref%
{E:lemma1}) and (\ref{E:Elbert}),
\begin{equation}
j_{m,k}-j_{m,l}^{\prime }\geq \sqrt{2}+\pi (k-l)  \notag
\end{equation}%
and (\ref{E:thm}) holds if one chooses $C=1$ (this is clearly not the
sharpest bound!). In the case when $l>k,$ using (\ref{E:lemma2}) and (\ref%
{E:Elbert}) we obtain
\begin{equation}
j_{m,l}^{\prime }-j_{m,k}\geq 1+\pi (l-k-1),  \notag
\end{equation}%
and (\ref{E:thm}) again holds if one chooses $C=1.$ This proves (\ref{E:thm}%
) for all roots of functions $J_{m}$ with $m\geq 1$. Now, let us consider
the function $J_{0}(x).$ Due to the asymptotic behavior of the roots (see
equations (\ref{E:asympt1}) and (\ref{E:asympt2})), there are constants $%
C_{0}$ and $C_{1}$ such that
\begin{align*}
\left\vert j_{0,k}-j_{m,n}\right\vert & \geq C_{0},\qquad k,n=1,2,3,... \\
j_{0,k}-j_{0,k}^{\prime }& \geq C_{1},\qquad k=1,2,3,... \\
j_{0,k+1}^{\prime }-j_{0,k}& \geq C_{1},\qquad k=1,2,3,...
\end{align*}%
with $C_{1}\leq C_{0}/2.$ Then the following inequality holds:%
\begin{equation}
\left\vert j_{0,k}-j_{0,l}^{\prime }\right\vert \geq C_{1}\left\vert
2k-2l+1\right\vert ,\qquad k,l=1,2,3,...  \notag
\end{equation}%
Set $C= \operatorname{min}(1,C_{1})$ and~(\ref{E:thm}) holds for all values of $%
m,l,k$ of interest.
\end{proof}

%------------------------------------------------------------------------------

\newpage
% \section*{References}

\end{document}